\documentclass[review,hidelinks,onefignum,onetabnum]{siamart220329}
\nolinenumbers

\usepackage{answers}
\usepackage{setspace}
\usepackage{enumitem}
\usepackage{multicol}
\usepackage{mathrsfs}
\usepackage{listings}
\usepackage{xspace}

\usepackage{amsfonts} %
\usepackage{latexsym}
\usepackage{color}
\usepackage{graphics} 
\usepackage{float}
\usepackage{subfigure}
\usepackage{epsfig} 
\usepackage{epstopdf}
\usepackage{caption}
\usepackage{longtable}
\usepackage{afterpage}
\usepackage{array,booktabs,enumitem}
\usepackage{algorithm}
\usepackage{algorithmicx}  
\usepackage{algpseudocode}  
\usepackage{mathrsfs}
\usepackage{bm}
\usepackage{tikz}
\usepackage{hyperref}
\usepackage{cleveref}
\usepackage{multirow}
\newsiamremark{remark}{Remark}
\newsiamremark{example}{Example}

\newcommand{\newsiamassump}[2]{
  \theoremstyle{plain}
  \theoremheaderfont{\normalfont \bf}
  \theorembodyfont{\normalfont}
  \theoremseparator{}
  \theoremsymbol{}
  \newtheorem{#1}{#2}
}

\newsiamassump{assumption}{Assumption}

\newcommand{\newsiamprob}[2]{
  \theoremstyle{plain}
  \theoremheaderfont{\normalfont \bf}
  \theorembodyfont{\normalfont}
  \theoremseparator{:}
  \theoremsymbol{}
  \newtheorem{#1}{#2}
}

\newsiamprob{test}{Problem} 

\usepackage[colorinlistoftodos,bordercolor=orange,backgroundcolor=orange!20,linecolor=orange,textsize=scriptsize]{todonotes}
\usepackage{frenchineq}
\newcommand{\veryshortarrow}[1][3pt]{\mathrel{%
   \hbox{\rule[\dimexpr\fontdimen22\textfont2-.2pt\relax]{#1}{.4pt}}%
   \mkern-4mu\hbox{\usefont{U}{lasy}{m}{n}\symbol{41}}}}

\makeatletter

\setbox0\hbox{$\xdef\scriptratio{\strip@pt\dimexpr
    \numexpr(\sf@size*65536)/\f@size sp}$}

\newcommand{\scriptveryshortarrow}[1][3pt]{{%
    \hbox{\rule[\scriptratio\dimexpr\fontdimen22\textfont2-.2pt\relax]
               {\scriptratio\dimexpr#1\relax}{\scriptratio\dimexpr.4pt\relax}}%
   \mkern-4mu\hbox{\let\f@size\sf@size\usefont{U}{lasy}{m}{n}\symbol{41}}}}

\makeatother

\newcommand{\fromsource}{_{\mathrm{s}\veryshortarrow}}
\newcommand{\todest}{_{\veryshortarrow\mathrm{d}}}

\newcommand{\source}{x_{\textrm{src}}}
\newcommand{\dest}{x_{\textrm{dst}}}

\newcommand{\sourceset}{\mathcal{K}_{\textrm{src}}}
\newcommand{\destset}{\mathcal{K}_{\textrm{dst}}}

\newcommand{\optsourceset}{\mathcal{X}_{\textrm{src}}}
\newcommand{\optdestset}{\mathcal{X}_{\textrm{dst}}}

\newcommand{\Fine}{\text{\sc Fine}\xspace}

\newcommand{\Coarsegrid}{\text{\sc Coarse-grid}\xspace}

\newcommand{\Far}{\text{\sc Far}\xspace}
\newcommand{\Accepted}{\text{\sc Accepted}\xspace }
\newcommand{\Narrow}{\text{\sc NarrowBand}\xspace } 
\newcommand{\Active}{\text{\sc Active}\xspace }
\newcommand{\Start}{\text{\sc Start}\xspace } 
\newcommand{\End}{\text{\sc End}\xspace } 

\newcommand{\lowerboundf}{\underline{f}}
\newcommand{\upperboundf}{\overline{f}}

\newcommand{\spacecom}{\mathcal{C}_{spa}}
\newcommand{\computcom}{\mathcal{C}_{comp}}

\newcommand{\R}{\mathbb{R}}
\newcommand{\A}{{\mathcal A}}

\newcommand{\Omegab}{\overline{\Omega}}

\newcommand{\F}{{\mathcal F}}

\newcommand{\Cc}{{\mathcal C}}

\newcommand{\NEW}[1]{{\em #1}}

\newcommand{\Oc}{\mathcal{O}}
\newcommand{\gT}{\partial_{\mathrm t} \Oc}
\newcommand{\Oetab}{\overline{\mathcal{O}}_{\eta}}
\newcommand{\Oeta}{\mathcal{O}_{\eta}}
\newcommand{\Od}{\mathcal{O}}

\newcommand{\Oetamu}{\mathcal{O}_{\eta}^\mu}
\newcommand{\Oetabmu}{\overline{\mathcal{O}_{\eta}^{\mu}}}

\newcommand{\etaH}{\eta_H}
\newcommand{\epsH}{\varepsilon_H}
\newcommand{\argmin}{\operatorname{Argmin}}

\newcommand{\myfrac}[2]{{#1}/{#2}}



\renewcommand{\theenumi}{\roman{enumi}}%

\headers{A multilevel fast-marching method For The Minimum time problem}{Marianne Akian, St\'ephane Gaubert, and Shanqing Liu}

\title{A multilevel fast-marching method For The Minimum time problem\thanks{Submitted \today.
	}}

\author{Marianne Akian\thanks{Inria and CMAP, \'Ecole polytechnique, IP Paris, CNRS
    (\email{Marianne.Akian@inria.fr, Stephane.Gaubert@inria.fr}).}
  \and St\'ephane Gaubert\footnotemark[2]
	\and Shanqing LIU\thanks{CMAP, \'Ecole polytechnique, IP Paris, CNRS, and Inria
		(\email{Shanqing.Liu@polytechnique.edu}). }}

\begin{document}
 
\maketitle

\begin{abstract}
	We introduce a new numerical method to approximate the solutions of a class of stationary Hamilton-Jacobi (HJ) partial differential equations arising from minimum time optimal control problems. We rely on nested grid approximations, and look for the optimal trajectories by using
	the coarse grid approximations to reduce the search space %
	in fine grids.
	This provides an infinitesimal version of the ``highway hierarchy'' method which has been developed to solve shortest path problems
(with discrete time and discrete state).
	We obtain, for each level, an approximate value function on a sub-domain of the state space. 
        We show that the sequence %
	obtained in this way does converge to the viscosity solution of the HJ equation.
Moreover, for our multi-level algorithm, if $0<\gamma \leq 1$ is the convergence rate of the classical numerical scheme, then the number of arithmetic operations needed to obtain an error in $O(\varepsilon)$ is in  $\widetilde{O}(\varepsilon^{-\theta })$, with $\theta< \frac{d}{\gamma}$,
to be compared with $\widetilde{O}(\varepsilon^{-d/ \gamma})$ for ordinary grid-based methods. 
Here $d$ is the dimension of the problem, $\theta $ depends on 
$d,\gamma$ and on the ``stiffness" of the value function around optimal trajectories, and the notation $\widetilde{O}$ ignores logarithmic factors. 
In particular, in typical smooth cases, one has $\gamma=1$ and $\theta=(d+1)/2$. 
\end{abstract}



\maketitle
\section{Introduction}
\subsection{Motivation and context}
We consider a class of optimal control problems, consisting of finding the minimum traveling time between two given sets in a given domain.
 Such optimal control problems are associated to a stationary Hamilton-Jacobi (HJ) equation via the Bellman dynamic programming principle  (see for instance~\cite{flemingsoner}). In particular, the value function is characterized as the solution of the associated HJ equation in the viscosity sense. 
 Problems with state constraints can be addressed with the notion of constrained viscosity solution~\cite{soner1986optimal1}. 

Various classes of numerical methods have been proposed for these problems. 
The \textit{Finite difference schemes} are based on a direct discretization of the HJ equation, see for instance~\cite{Crandall1984TwoAO}. High order schemes, such as ENO and WENO, have been developed~\cite{shu1988efficient,jiang1996efficient}. %
The \textit{Semi-Lagrangian schemes}, as in \cite{falcone2013semi}, arise by applying the Bellman dynamic programming principle to the discrete time optimal control problem obtained after an Runge-Kutta time-discretization of the dynamics.
Other numerical methods, using approximations on grids, have been developed, including the \textit{Discontinuous Galerkin schemes} (as in~\cite{bokanowski2014discontinuous}) and  the \textit{Finite element schemes} (as in~\cite{jensen2013convergence, smears2016discontinuous}). 
Max-plus based discretization schemes have also been deeloped~\cite{fleming2000max, akian2008max}. In some of these cases, the discretized system of equations can be interpreted as the dynamic programming equation of a deterministic or stochastic optimal control problem \cite{kushner2001numerical} with discrete time and state space.

Traditional numerical methods apply iterative steps to solve the discretized stationary HJ equation, for instance value iteration or policy iteration. At each step, the value function is computed in the {whole} discretization grid. 
In the particular case of the shortest path problem (with discrete time and state space), with a nonnegative cost function,
one can solve the stationary dynamic programming equation by Dijkstra's algorithm.
The fast marching method introduced by Sethian~\cite{sethian1996fast} and by Tsitsiklis~\cite{tsitsiklis1995efficient}, then further developed in particular in~\cite{sethian2003ordered,  cristiani2007fast, carlini2011generalized, clawson2014causal,Mir14a,mirebeau2019riemannian}, provides an analogue of Dijkstra's method for the HJ equation
of shortest path problems
with {\em continuous} time and state; it relies on a monotone finite difference or semi-lagrangian discretization scheme.
It is called "single pass" because at every point of the discretization grid, the value is computed at most $k$ times, where $k$ is a bound not related to the discretization mesh.
This approach was initially developed to solve the front propagation problem, then extended to more general stationary Hamilton-Jacobi equations~\cite{sethian2003ordered}. It takes advantage of the property that the evolution of the region behind a ``propagation front'' is monotonely increasing, the so called the ``causality'' property.
Fast marching remains a grid-based method, and hence still suffers from the "curse of dimensionality".
Indeed, the number of grid nodes grows exponentially with the dimension $d$, making the reading, writing, storing and computation untractable even on modern computers. 
Several types of discretizations or representations have been developed recently to overcome the curse of dimensionality for HJ equations. One may cite the sparse grids approximations%
~\cite{zbMATH06176972},  %
the tensor decompositions%
~\cite{zbMATH07364328,zbMATH07547920}, %
the deep learning techniques~\cite{zbMATH07508503,bokanowski2023neural},
the curse-of-dimensionality free
max-plus numerical method of McEneaney \cite{%
	Mc2007}, 
and the Hopf formula based method of~\cite{zbMATH06626046},
apply to problems with specific structures.

Another way to overcome the curse of dimensionality is to replace the general problem of solving the HJ equation by the one of computing the optimal trajectories between two given points. The latter problem can be solved, under some convexity assumptions, by the Pontryagin Maximum Principle approach (see~\cite{raymond1998pontryagin} and the reference there).  %
Another method is the stochastic dual dynamic programming (SDDP), see~\cite{girardeau2015convergence} and reference there, in which the value function is approximated by a finite supremum of affine maps, and thus can be computed efficiently by linear programming solvers. In the absence of convexity assumptions, these methods may only lead to a local minimum. In that case, more recent methods consist in exploiting the structure of the problem, in order to reduce the set of possible trajectories among which the optimization is done, for instance in~\cite{alla2019efficient,bokanowski2022optimistic}.

This work is inspired by the recent development of "Highway Hierarchies"~\cite{delling2006highway, sanders2012engineering} for the (discrete) shortest path problems. Highway hierarchies allow one to accelerate Dijkstra's algorithm by computing multi-level aggregated graphs using ``highways''  in which the optimal paths should go through. This provides the exact shortest path between two given points, in a much shorter time.
Highway hierarchies are specially useful when numerous queries have to be solved quickly, on a fixed graph
(a typical use case is GPS path planning). Acceleration methods for Dijkstra's algorithm were also used for continuous shortest path problems in~\cite{peyre2008heuristically, clawson2014causal}, in order to accelerate fast marching method.

\subsection{Contribution}
In this paper, we intend to find the optimal trajectories between two given sets, for the minimal time problem. 
Our approach combines the idea of highway hierarchy with the one of multi-level
grids, in order to reduce the search space. 
We compute approximate ``geodesics'' by using a coarse grid.
Then, we perform a fast marching in a finer grid, restricted to a neigborhood of the highways. This method is iterated, with finer and finer grids, and smaller
and smaller neighborhoods, until the desired accuracy is reached. 
It can be seen as a highway hierarchy method for the eikonal equation, comparing geodesic with highways and multi-level grids with multi-level graphs. 

To carry out our analysis, we suppose that the fast marching method with a mesh step $h$ has an error in $O(h^\gamma)$, and the same bound remains valid under suitable restrictions of the state space (\Cref{assup_error}). We show (in~\Cref{theo-corr-multi}) that the final approximation error for the value function, restricted to a neighborhood
of the optimal trajectory,
is as good as the one obtained by discretizing directly the whole domain with the finest grid. 
Moreover, we show (in~\Cref{complexity_1}) that the number of  elementary operations and the size of the memory needed to get an error of $\varepsilon$
are considerably reduced. 
Indeed, %
recall that the number of arithmetic operations
of fast marching methods
is in $\widetilde{O}(\varepsilon^{-\frac{d}{\gamma}}K_d)$,
where $d$ is the dimension of the problem, 
$K_d \in [2d,L^d]$ for some constant $L$ depending on the 
diameter of discrete neighborhoods 
and the notation $\widetilde{O}(\cdot)$ ignores the logarithmic factors (see~\Cref{sec-conv}).
For our multi-level method, with suitable parameters,
the number of arithmetic operations
is in $\widetilde{O}(C^d \varepsilon^{-\theta})$, where $\theta < \frac{d}{\gamma}$ further depends on a geometric parameter $0< \beta \leq 1$, defined in \Cref{distance_beta} which measures the ``stiffness" of the value function around optimal trajectories. 
In typical situations in which the value
function is smooth with a nondegenerate Hessian in the neighborhood
of an optimal trajectory, one has $\beta=1/2$ and $\theta = \frac{1}{2}+ (1 - \frac{\gamma}{2})\frac{d}{\gamma}$. 
In some exceptional cases, we have $\beta=\gamma=1$, and then the complexity bound reduces to %
$\widetilde{O}(C^d\varepsilon^{-1}) $, which is of the same order as for 
one dimensional problems. 
The theoretical complexity bounds are consistent with 
the complexity observed in numerical experiments.
However, as said above, our theoretical bounds
hold under \Cref{assup_error},
requiring the error of the fast marching method to be uniform
when the domain varies in a certain class,
whereas standard results suppose that the domain is fixed and focus on the order of convergence. The issue of the explicit dependence of this convergence in the shape of the domain, for state constraints problems, seems to have been little  explored, see the discussion in~\Cref{remark_error}. This issue
is of independent interest, and we leave it as an open question.

This paper is organized as follows. In \Cref{Pre}, we give preliminary results on the HJ equation and the minimum time optimal control problem. In \Cref{sec-cotra}, we establish key results justifying the restriction of the state space to a neighborhood of optimal trajectories. 
In \Cref{sec-algo}, we describe our algorithm and establish
its correctness.
Complexity estimates are given in \Cref{sec-conv}. 
Numerical examples are given in~\Cref{sec-test}. 

\section{Hamilton-Jacobi equation for the Minimum Time Problem} \label{Pre}

\subsection{The Minimum Time Problem }
Let $\Omega$ be an open, bounded domain in $\R^{d}$. Let $S_{1}$ be the unit sphere in $\R^{d}$, i.e., $S_{1} = \{ x \in \R^{d}, \| x\| =1  \} $ where $\|\cdot\|$ denotes the Euclidean norm.
Let $\mathcal{A} = \{ \alpha : [0,\infty) \mapsto S_{1} : \alpha(\cdot) \text{ is measurable} \}  $ denote the set of controls, every $\alpha \in \A$ is then the unit vector determining the direction of motion. 
We denote by $f$ the speed function, and assume the following basic regularity properties:
\begin{assumption} \label{assp1}$ $
      \begin{enumerate}
      \item $f : \Omegab\times S_1 \mapsto (0,\infty)$  is continuous.	
      \item $f$ is bounded on $\Omegab\times S_1$, i.e., $\exists M_f >0$ s.t. $\|f(x,\alpha) \| \leq M_f$, $\forall x \in \Omegab, \forall \alpha \in S_1$.
      \item There exists constants $L_{f}, L_{f,\alpha}>0$ s.t. $ |f(x,\alpha) - f(x^{'},\alpha)| \leq L_{f}| x - x^{'} |, \forall \alpha \in S_1, \forall x, x^{'}\in \Omega$ and  $| f(x,\alpha) - f(x,\alpha^{'})| \leq L_{f,\alpha}|\alpha - \alpha^{'}|, \forall x \in \Omega, \forall \alpha, \alpha^{'} \in S_1$.
      \end{enumerate}
\end{assumption}
It's worth noting that the original fast marching method presented in~\cite{tsitsiklis1995efficient,sethian1996fast} focuses on cases where $f(x,\alpha) = f(x)$, meaning the speed function does not depend on the direction, a property known as ``isotropy''.  
Furthermore,~\Cref{assp1} applies to both Riemannian and Finslerian geometry but excludes sub-Riemannian systems and, more broadly, the dynamics of nonholonomic systems.

Let $\sourceset$ and $\destset$ be two disjoint compact subsets of $\Omega$ 
(called the \NEW{source} and the \NEW{destination} resp.).
Our goal is to find the minimum time necessary to travel from $\sourceset$ to $\destset$, and the optimal trajectories between $\sourceset$ and $\destset$, together with the optimal control $\alpha$. 
So, we consider the following optimal control problem:
\begin{equation}\label{problem}
	\begin{aligned}
		&\inf \ \tau \geq 0\qquad 
		s.t. \left\{
			\begin{aligned}
				& \dot{y}(t) = f(y(t),\alpha(t))\alpha(t), \ \forall t\in[0,\tau] \enspace , \\
				& y(0)\in\sourceset, \ y(\tau)\in\destset \enspace, \\
				& y(t) \in \overline{\Omega}, \ \alpha(t) \in \A, \ \forall t \in [0,\tau] \enspace .
			\end{aligned} \right.
	\end{aligned}
\end{equation}
\subsection{HJ Equation for the Minimum Time Problem}\label{subsec-hj1} A well known sufficient and necessary optimality condition for the above problem is given by the Hamilton-Jacobi-Bellman equation, which is deduced from the dynamic programming principle. Indeed, one can consider the following controlled dynamical system:  
\begin{equation}
\label{dynmsys}
\dot{y} (t) = f(y(t), \alpha(t)) \alpha(t), \ \forall t \geq 0 \ ,\qquad
y(0) = x \ .
\end{equation}
We denote by $y_{\alpha}(x;t)$ the solution of
\eqref{dynmsys} with $\alpha \in \mathcal{A}$ 
and $y_\alpha(x;s) \in \overline{\Omega}$ for all $0\leq s \leq t$. More precisely, we restrict the set of control trajectories so that the state $y$ stays inside the domain $\overline{\Omega}$, considering the following set of controls: 
\begin{equation}\label{controlset}
	\mathcal{A}_{\Omega,x} := \{ \alpha \in \mathcal{A} \ | \  y_\alpha(x;s) \in \overline{\Omega}, \text{ for all }  s \geq 0  \} \enspace .
\end{equation}
 By~\Cref{assp1}, $\A_{\Omega,x} \neq \emptyset$. In other words, the structure of the control set $A_{\Omega,x}$ is adapted to the state constraint ``~$y\in \overline{\Omega}$~''. 
Let us define the cost functional by: 
\begin{equation}\label{cost-func}
J\todest(\alpha(\cdot),x) = \inf \{ \tau \geq 0\mid y_{\alpha}(x;\tau) \in \destset \} \enspace,
\end{equation}
in which "$\todest$" stands for ``arrival to destination''. The value function $T\todest : \bar{\Omega} \mapsto \R \cup \{+\infty\}$ is defined by %
\begin{equation}\label{Mintime}
T\todest(x) = \inf_{\alpha \in \mathcal{A}_{\Omega,x}} J\todest(\alpha(\cdot),x) \ . %
\end{equation}
Then, restricted to $\overline{\Omega \setminus \destset}$, informally $T\todest$ is a solution of the following state constrained Hamilton-Jacobi-Bellman equation:
\begin{equation} \label{hjes1}
\left\{
\begin{aligned}
&-(\min_{\alpha \in S_1} \{ (\nabla T\todest(x) \cdot \alpha) f(x,\alpha) \} +1) =0, \enspace & x\in \Omega \setminus \destset \ ,\\
&T\todest(x) = 0 \enspace \enspace & x\in \partial \destset \ .
\end{aligned}
\right.
\end{equation}
In order to relate the minimum time function and the HJB equation,  
 we use the following definition  for 
the viscosity solution of the following state constrained HJ equation 
with continuous hamiltonian $F: \R^d\times \R\times \R^d\to \R$,
open bounded domain $\Oc\subset \R^d$, 
and ``target'' $\gT\subset \partial \Oc$: 
\begin{equation}\label{constrainHJ}
	\tag{\text{$SC(F,\Oc,\gT )$}}
	\left\{
	\begin{aligned}
	& F(x, u(x), Du(x)) = 0, \enspace & x \in \Oc\ , \\
	&F(x, u(x), Du(x)) \geq 0, \enspace & x \in \partial\Oc\setminus(\gT)
 \ ,  \\
	&u(x) = 0, \enspace & x \in \gT \ .
	\end{aligned}
	\right.
	\end{equation}
When there is no target set, that is $\gT=\emptyset$,
this definition corresponds to the one
introduced first by Soner in~\cite{soner1986optimal1}
(see also \cite{bardi2008optimal}), the case with a nonempty
closed target set $\gT$
is inspired by the results of \cite{capuzzo-lions}.
\begin{definition}[compare with \protect{\cite{soner1986optimal1,bardi2008optimal,capuzzo-lions}}]
\label{constraint_vis} 
	Let $u: \overline{\Oc} \to \R$ be continuous.
	\begin{enumerate}
		\item The function $u$ is a viscosity subsolution 
		of \eqref{constrainHJ} if %
		for every test function $\psi \in \Cc^{1}(\overline{\Oc} )$, for all local maximum points $x_0 \in \overline{\Oc}$ of the function $u-\psi$, we have
                $F(x_0,u(x_0),D\psi (x_0) ) \leq 0$
                if 
		$x_0\in \Oc$ and
                $u(x_0)\leq 0$
                if
		$x_0\in \gT$.
	      \item The function $u$ is a viscosity supersolution of \eqref{constrainHJ} if for every test function $\psi \in \Cc^{1}(\overline{\Oc})$, for all local minimum points $x_0 \in \overline{\Oc}$ of the function $u-\psi$, we have
                $F(x_0,u(x_0),D\psi (x_0) ) \geq 0$ if
		$x_0\not\in \gT$
                and $u(x_0)\geq 0$ otherwise.
		\item The function $u$ is a viscosity solution of \eqref{constrainHJ} if and only if it is a viscosity subsolution and supersolution of \eqref{constrainHJ}.
	\end{enumerate}
\end{definition}

A basic method in the studies of the above system (see \cite{vladimirsky2006static}, \cite[Chapter-IV]{bardi2008optimal}) is the change of variable:
\begin{equation} \label{changevar}
	v\todest(x) = 1 - e^{-T\todest(x)} \enspace,
\end{equation}
which was first used by Kruzkov~\cite{kruvzkov1975generalized}.
The function $v\todest(x)$ is bounded and Lipschitz continuous. The minimum time is recovered
by 
$T\todest(x) = - \log(1 - v\todest(x))$.  

In fact, consider a new control problem associated to the dynamical system \eqref{dynmsys}, and the discounted cost  functional defined by
\begin{equation}\label{disc-cost}
J\todest^{'}(\alpha (\cdot),x) = \inf \left\{ \int_{0}^{\tau} e^{-t} dt \mid \tau\geq 0 , \ y_{\alpha}(x;\tau) \in \destset  \right\} \enspace ,
\end{equation}
for $\alpha \in \mathcal{A}_{\Omega,x}$. Then, the value function $v$ of the control problem given by
\begin{equation}\label{value-disc}
v (x) = \inf_{\alpha \in \mathcal{A}_{\Omega,x}} J\todest^{'}(\alpha(\cdot),x)
\end{equation}
 coincides with $v\todest$ in \eqref{changevar}.
Define now 
\begin{equation}
\label{defF}
F(x,r,p) = -\min_{\alpha \in S_{1}} \{ p \cdot f(x,\alpha)\alpha + 1 - r \}
\enspace .\end{equation}
 This Hamiltonian corresponds to the new control problem {\rm (\ref{dynmsys},\ref{disc-cost},\ref{value-disc})}, and the restriction of the value function $v\todest$ to $\overline{\Omega \setminus \destset}$ is a viscosity solution of the state constrained HJ equation $SC(F,\Omega\setminus \destset , \partial\destset)$.

The uniqueness of the solution of Equation $SC(F,\Omega\setminus \destset , \partial\destset)$ in the viscosity sense and the equality of this solution with the value function need not hold if the boundary condition is not well defined. 
When the target set is empty,
Soner \cite{soner1986optimal1} introduced sufficient conditions for the uniqueness of the viscosity solution
of \eqref{constrainHJ} and the equality with the corresponding value function.
One of these conditions involves the dynamics of the controlled process on 
$\partial \Oc$, the so-called ``inward-pointing condition'' (
see \cite[(A3)]{soner1986optimal1}), which is automatically
satisfied when $F$ is as in \eqref{defF}, and $f$ satisfies \Cref{assp1}
with $\Oc$ instead of $\Omega$. 
Similar conditions are proposed in \cite{capuzzo-lions}.
We state below the result of \cite{capuzzo-lions}
with the remaining conditions, and for a general open bounded domain $\Oc$,
instead of $\Omega$, 
as we shall need such a result in the sequel.

\begin{theorem}[Corollary of \protect{\cite[Th.\ IX.1, IX.3 and X.2]{capuzzo-lions}, see also \cite{soner1986optimal1}}]\label{th_soner}
Let $\Oc$ be an open domain of $\R^d$, let 
$\gT\subset \partial \Oc$ be compact, and assume that
 $\partial \Oc\setminus \gT$ is of class $\Cc^1$.
Let $F$ be as in \eqref{defF} with $f$ satisfying \Cref{assp1}
with $\Oc$ instead of $\Omega$.
Then the comparison principle holds for $SC(F,\Oc , \gT)$, i.e., any viscosity subsolution is upper bounded by any viscosity supersolution. In particular, the viscosity solution is unique. Moreover, it coincides with 
the value function $v\todest$
in \eqref{value-disc} of the optimal control problem with dynamics 
\eqref{dynmsys} and
criteria \eqref{disc-cost},
in which $\Omega$ and $\destset$ are replaced by $\Oc$ and $\gT$, respectively.
\end{theorem}

We should also mention the recent work of \cite{ bokanowski2011deterministic}, 
which characterized the value function of the state constrained problems without any
controllability assumptions.

Once $SC(F,\Omega\setminus \destset , \partial\destset)$ is solved, one can easily get the value of the original minimum time problem by computing the minimum of $v\todest(x)$ over $\sourceset$. We shall denote the set of minimum points by $\optsourceset$, i.e., 
Since $v\todest$ is continuous (by \Cref{th_soner}) and $\sourceset$ is compact,
we get that $\optsourceset$ is a nonempty compact set.
\subsection{HJ Equation in Reverse Direction}\label{sec_reverse} We shall also use an equivalent optimality condition for the minimum time problem \eqref{problem}, obtained by applying the dynamic programming principle in a reverse direction. 

Let us consider the following controlled dynamical system:
\begin{equation}
\label{dynmsys_t}
\dot{\tilde{y} } (t) = -f(\tilde{y}(t), \tilde{\alpha}(t)) \tilde{\alpha}(t), \ \forall t \geq 0 \ ,\qquad
\tilde{y}(0) = x \ .
\end{equation}
We denote by $\tilde{y}_{\tilde{\alpha}}(x;t)$ the solution of the above dynamical system \eqref{dynmsys_t} with $\tilde{\alpha}(t) = \alpha(\tau-t) \in \A$, for all $t\in[0,\tau]$. Then automatically $\tilde{y}(t) = y(\tau - t)$ with $y$ as in \eqref{dynmsys}.
We denote the state constrained control trajectories for this new problem by $\tilde{\A}_{\Omega,x}$:
\begin{equation}
	\tilde{\A}_{\Omega,x} = \{ \tilde{\alpha} \in \A \ | \ \tilde{y}_{\tilde{\alpha}}(x;s) \in \overline{\Omega}, \text{ for all } s \geq 0 \}\ .
\end{equation}
Consider the following cost functional and associated value function:
\begin{eqnarray}
\label{cost_s}
  &J\fromsource(\tilde{\alpha}(\cdot),x) = \inf\{ \tau \geq 0\mid \tilde{y}_{\tilde{\alpha}}(x;\tau) \in \sourceset \} \, ,\\
  &T\fromsource(x) =\inf_{\tilde{\alpha}\in \tilde{\A}_{\Omega,x}} J\fromsource(\tilde{\alpha}(\cdot),x) \in \R \cup \{+\infty\}\enspace,\nonumber
\end{eqnarray}
where "$\fromsource$" means ``from source''.
Using similar change of variable technique as above, we consider $v\fromsource(x) = 1 - e^{-T\fromsource(x)}$. This is the value function of a new optimal control problem with dynamics \eqref{dynmsys_t}, and is defined by 
\begin{equation}\label{value_s_v}
	v\fromsource(x) = \inf_{\tilde{\alpha}\in \tilde{\A}_{\Omega,x}} \inf \left\{ \int_{0}^{\tau} e^{-t} dt \mid \tau\geq 0,\ \tilde{y}_{\tilde{\alpha}}(x;\tau) \in \sourceset \right\} \ .
\end{equation}
Then, $v\fromsource$ is the viscosity solution of
the state constrained HJ equation $SC(\tilde{F},\Omega\setminus\sourceset,\partial \sourceset)$, where $\tilde{F}(x,r,p) = F(x,r,-p)$. %
By doing so, to solve the original minimum time problem \eqref{problem}, one can also solve the equation $SC(\tilde{F},\Omega\setminus\sourceset,\partial \sourceset)$ to get $v\fromsource$, 
and then compute the minimum of $v\fromsource(x)$ over $\destset$. We shall also denote by $\optdestset$ the set of minimum points, i.e. 
	 \(\optdestset = 
\argmin_{x \in \destset}  v\fromsource(x) 
\).
Again, as for $\optsourceset$,  we get that $\optdestset$ 
 is a nonempty compact set.

\section{Reducing the State Space of the Continuous Space Problem}\label{sec-cotra}
\subsection{Geodesic points}
We are interested in the set of points of optimal trajectories:
\begin{definition}
	For every $x\in \Omegab$, 
	We say that $y_{\alpha^{*}}(x;\cdot):[0,\tau] \mapsto \Omegab$ is an optimal trajectory with associated optimal control $\alpha^{*}$ for Problem 
	{\rm (\ref{dynmsys},\ref{disc-cost},\ref{value-disc})}, if the minimum in \eqref{value-disc} is achieved in $\alpha^{*}$. We denote by $\Gamma^{*}_x$ the set of {\em geodesic points} starting from $x$, i.e., 
	\[\Gamma^{*}_x = \{ y_{\alpha^{*}}(x;t) \ | \ t\in[0,\tau] , \ {\alpha}^* \text{ optimal } \}\enspace.\]
        Similarly,
        $\tilde{\Gamma}_x^{*}= \{ \tilde{y}_{\tilde{\alpha}^{*}}(x;t) \ | \ t \in [0,\tau], \ \tilde{\alpha}^* \text{ optimal } \}$ denotes the set of geodesic points starting from $x$ in the reverse direction,
        as per \Cref{sec_reverse}.
\end{definition}
\begin{proposition}\label{geodesicpoints}
We have
	$\cup_{x\in \optsourceset} \{ \Gamma^*_x \} = \cup_{x\in\optdestset} \{ \tilde{\Gamma}^*_x \}$, 
and if the latter set is nonempty, then
\( \inf_{x \in \sourceset} v\todest(x) = \inf_{x  \in \destset} v\fromsource(x)\), \( \inf_{x \in \sourceset} T\todest(x) = \inf_{x  \in \destset} T\fromsource(x)\). 
	\end{proposition}
\begin{proof}
	Let $y_{\alpha^*}(\source;\cdot) : [0,\tau^*] \mapsto \Omegab$  be an optimal trajectory for the problem {\rm (\ref{dynmsys},\ref{disc-cost},\ref{value-disc})}, with $\source \in \sourceset$ and the optimal control $\alpha^*$, assuming one exists. Let us denote $\dest:=y_{\alpha^*}(\source;\tau^*)  \in \destset$. Consider the problem in reverse direction starting at $\dest$, and the control $\tilde{\alpha}^*$ such that $\tilde{\alpha}^*(s) = \alpha^*(\tau^* - s)$, $\forall s \in [0,\tau^*]$. Then, the associated state at time $s$ is $\tilde{y}_{\tilde{\alpha}^*}(\dest;s) = y_{\alpha^*}(\source;\tau^* -s)$. In particular $\tilde{y}_{\tilde{\alpha}^*}(\dest;\tau^*)  = \source \in \sourceset$, and
the trajectory $\tilde{y}_{\tilde{\alpha}^*}(\dest;\cdot)$ arrives in $\sourceset$ 
at time $\tau^*$. By definition of the value function $v\fromsource$, we have
		$v\fromsource(\dest) \leq  
		 v\todest (\source)$, 
with equality holds if and only if $\tilde{\alpha}^*$ is optimal.

Let us assume that $\cup_{x\in \optsourceset} \{ \Gamma^*_x \}$ is nonempty, 
and take $\source \in \optsourceset$, such that $\Gamma^*_{\source}$ is
nonempty, we get 
	\begin{equation}\label{value_2direction}
		v\fromsource(\dest) \leq v\todest(\source)=\inf_{\source \in \sourceset} v\todest(\source) \ .
		\end{equation}
If the above inequality is strict, then there exists
a trajectory (not necessary optimal)
$\tilde{y}_{\tilde{\alpha}'}(\dest;\cdot)$  starting from $\dest$ and arriving
in $\source'\in\sourceset$ at time $\tau'< \tau^*$.
Then, the reverse trajectory $y_{\alpha'}(\source';\cdot)$ 
is starting from $\source'$ and arrives in $\dest$ at time $\tau'$,
and we get $v\todest(\source')= 1-e^{-\tau'}< v\todest(\source)$ which is 
impossible.
This shows the equality in \eqref{value_2direction} and that 
$\tilde{\alpha}^*$ is optimal, so $\tilde{\Gamma}^*_{\dest}$ is nonempty.
Also, if $\dest\notin\optdestset$,
 by the same construction applied to $\dest'\in \optdestset$,
we get a contradiction, showing that $\dest\in\optdestset$.
Hence $\cup_{x\in\optdestset} \{ \tilde{\Gamma}^*_x \}$ is nonempty,
and 
$y_{\alpha^*}(\source;s) =  \tilde{y}_{\tilde{\alpha}^*}(\dest;\tau^*-s)  \in  
\cup_{x\in\optdestset} \{ \tilde{\Gamma}^*_x\} . $ 
Since this holds for all optimal trajectories $y_{\alpha^*}$ starting in
any $\source \in \optsourceset$, we deduce that
$\cup_{x\in \optsourceset} \{ \Gamma^*_x \}\subset  \cup_{x\in\optdestset} \{ \tilde{\Gamma}^*_x \}$.
By symmetry, we obtain the equality, so the first equality of
the proposition.
Moreover, by the equality in \eqref{value_2direction}, 
we also get the second equality of the proposition. 
	\end{proof} 
From now on,
 we set $\Gamma^* =\cup_{x\in \optsourceset} \{ \Gamma^*_x \} = \cup_{x\in\optdestset} \{ \tilde{\Gamma}^*_x \} $, and call it \NEW{the set of geodesic points from $\sourceset$ to $\destset$}.
When $\Gamma^*$ is nonempty, using \Cref{geodesicpoints},
we shall denote by $v^*$ the following value:
\[ 
v^*:=\inf_{x \in \sourceset} v\todest(x) = \inf_{x 	\in \destset} v\fromsource(x)
\enspace . \] 
Once $v^*$ is obtained, we can get the minimum time by $\tau^* = -\log(1-v^*)$.

\begin{lemma}\label{pro-tra} Assume $\Gamma^{*}$ is non-empty. Then, we have 
\begin{equation}\label{opt-traj}
v^*= \inf_{y \in \Omegab} \{v\fromsource (y) + v\todest(y) - v\fromsource(y) v\todest(y) \}\enspace ,
\end{equation}
and the infimum is attained for every $x \in \Gamma^{*}$.
Moreover, if there exists an optimal trajectory between any two points of $\Omegab$, then $x$ is optimal in 
\eqref{opt-traj}, that is $(v\fromsource(x) + v\todest(x) - v\fromsource(x)v\todest(x) ) = v^*$, if and only if $x \in \Gamma^{*}$.
\end{lemma}
\begin{proof}
	We first notice that, by an elementary computation, \eqref{opt-traj} is equivalent to 
	\begin{equation}\label{opt-traj-time}
		\tau^* = \inf_{y \in \Omegab} \{ T\fromsource(y) + T\todest(y) \} \enspace.
		\end{equation}
Fix $x \in \Omegab$. 
For any $\tau>0$ and $\tilde{\alpha}\in {\tilde{\A}}_{\Omega,x}$ s.t.\ $\tilde{y}_{\tilde{\alpha}} (x;\tau) \in \sourceset$, denote $ \source = \tilde{y}_{\tilde{\alpha}} (x;\tau) \in \sourceset$. Then, by the change of variable $s = \tau -t$ and $\alpha(s)=\tilde{\alpha}(\tau-t)$, we have $\alpha \in \A_{\Omega,\source}$ and $y_{\alpha}(\source;\tau)=x$. This implies 
\begin{equation}\label{value_xfroms}
	T\fromsource(x) =  \inf \limits_{\tau>0,\tilde{\alpha}\in {\tilde{\A}}_{\Omega,x}\atop \tilde{y}_{\tilde{\alpha}} (x;\tau) \in \sourceset } \left\{ \tau \right\} = \inf \limits_{\tau >0, \source \in \sourceset, \alpha \in \A_{\Omega,\source} \atop y_{\alpha}(\source;\tau)=x} \left\{ \tau \right\} \enspace .
\end{equation}
Similarly for $T\todest$, we have 
\begin{equation}\label{value_xtodest}
	T\todest(x) = \inf_{\tau'>0,\alpha'\in {\mathcal A}_{\Omega,x} \atop  y_{\alpha'} (x;\tau^{'}) \in \destset }\left\{ \tau  \right\} = \inf_{\tau'>0, \dest\in  \destset, \alpha'\in {\mathcal A}_{\Omega,x} \atop  y_{\alpha'} (x;\tau^{'}) = \dest } \left\{  \tau  \right\} \enspace .
\end{equation}
Let use take $\tau >0$ 
such that $y_{\alpha}(\source;\tau)=x$, and $\tau'>0$ 
such that $ y_{\alpha'} (x;\tau^{'}) = \dest$.
Concatenating $\alpha$ stopped at time $\tau$ and $t\in [\tau,\infty)\mapsto \alpha'(t-\tau)$, we obtain $\alpha''\in {\mathcal A}_{\Omega,\source}$
such that $y_{\alpha''}(\source;\tau+\tau')=\dest$  and
the trajectory from $\source$ to $\dest$ is going through $x$ at time $\tau$.
 This implies
 \begin{equation}\label{sumoftau}
 	T\fromsource(x) + T\todest(x) \geq \tau^* \enspace ,
 	\end{equation}
where the inequality comes from the last
equality in \Cref{geodesicpoints}. Since the above inequality is an equality for $x\in\optdestset$ or $x\in\optsourceset$, we deduce \eqref{opt-traj-time}, and so~\eqref{opt-traj}.
If $x \in \Gamma^{*}$, there exist $\source\in\sourceset$, $\dest\in\destset$ and $\alpha\in {\mathcal A}_{\Omega,\source}$ such that $y_{\alpha}(\source;\tau^{*})=\dest$ and $y_{\alpha}(\source;\tau) = x$ for some $0\leq \tau\leq \tau^*$. 
Taking $\tau'=\tau^*-\tau$, we get an equality in \eqref{sumoftau}.

Let now $x\in \Omega$ be optimal in \eqref{opt-traj-time}. %
Assuming that there exists an optimal trajectory for each of the two minimum time problems starting from any point, %
that is the infimum in \eqref{value_xfroms} and \eqref{value_xtodest} are achieved.
So again concatenating $\alpha$ stopped at time $\tau=T\fromsource(x)$ and $t\in [\tau,\infty)\mapsto \alpha'(t-\tau)$, we obtain $\alpha''\in {\mathcal A}_{\Omega,\source}$ and
a trajectory from $\source$ to $\dest$ going through $x$ at time $\tau$ and arriving
at $\dest$ at time $\tau+\tau'$.
Using \eqref{sumoftau}, we get 
$\tau^*=T\fromsource(x) + T\todest(x)$, so $\alpha''$ is optimal, which shows 
$x\in \Gamma^*$.
\end{proof}

For easy expression, for every 
$x\in \overline{\Omega}$ and $v = (v\fromsource, v\todest)$, we denote
\begin{equation}\label{operF}
	\F_v(x) = v\fromsource(x) + v\todest(x) - v\fromsource(x) v\todest(x) \enspace.
\end{equation}

\subsection{Reduction of The State Space}
Let us now consider the open subdomain $\mathcal{O}_{\eta}$ of $\Omega$,
determined by a parameter $\eta > 0$, and defined as follows:
\begin{equation} \label{eqo}
	\Oeta = \{ x\in (\Omega \setminus ( \sourceset \cup \destset) )  \ | \ \F_v(x) < \inf_{y \in \Omega} \{ \F_v(y) + \eta \} \ \} \enspace.
\end{equation}
By \Cref{assp1}, we observe that $v\fromsource$, $v\todest$ are continuous in $\overline{\Omega}$, so does $\F_v$, thus the infimum in \eqref{eqo} is achieved by an element $y\in \overline{\Omega}$, and by \Cref{pro-tra}, it is equal to $v^*+\eta$.
This also implies that $\Oeta$ is an open set.
Since $\eta>0$, we also have that $\Oeta$ is nonempty.
\begin{proposition}\label{newboundary} Under \Cref{assp1}, 
and assuming that $\Gamma^*$ is nonempty, we have
\(\optsourceset \subseteq (\partial \Oeta) \cap (\partial \sourceset),  \ \optdestset \subseteq (\partial \Oeta) \cap (\partial \destset), 
\; \text{and}\;\Gamma^* \subset \Oetab \quad \forall \eta > 0 \).
	\end{proposition}
\begin{proof}
Let us first notice that,  since $\F_v$ is continuous,
we have $\Oetab \supseteq \{ x\in \overline{(\Omega\setminus (\sourceset \cup \destset) )}  \mid \F_v(x) <  \inf_{y \in \Omega} \{ \F_v(y) + \eta \}$. Moreover, since 
$\sourceset$ and $\destset$ are disjoint compact subsets of $\Omega$,
then  $\partial \Oeta \supseteq \{ x\in \partial\Omega\cup\partial\sourceset \cup \partial \destset   \ | \ \F_v(x) <  \inf_{y \in \Omega} \{ \F_v(y) + \eta \} \}$.

	By the dynamic programming principle, and since the cost in
\eqref{cost_s} is $1$ (so positive),
and $\sourceset$ and $\destset$ are disjoint,
we have 
	\( v\todest (x) > \inf_{y \in \partial \sourceset} v \todest(y)\), 
for all $x$ in the interior of $\sourceset$ (any trajectory starting in
$x$ need to go through $\partial \sourceset$).
This implies that $\Gamma^*$ does not intersect the interior of $\sourceset$, 
and similarly  $\Gamma^*$ does not intersect the interior of $\destset$,
so
$\Gamma^*\subseteq \overline{(\Omega\setminus (\sourceset \cup \destset) )}$.
Moreover, $\optsourceset \subseteq \sourceset\cap \Gamma^* \subseteq \partial \sourceset$ and $\optdestset \subseteq \destset\cap \Gamma^* \subseteq \partial \destset$.

Now for all $x\in \Gamma^*$, we have $\F_v(x) = v^* = \inf_{y\in \Omegab} \F_v(y),$
so $x\in \Oetab$, showing that $\Gamma^*\subseteq  \Oetab$.
Let us now take $x\in \optsourceset$. We already shown that
$\optsourceset \subseteq \partial \sourceset$,
and we also have $\optsourceset\subset \Gamma^*$,
so $\F_v(x) = v^*$. All together, this  implies that 
$x\in \partial \Oeta$.
The same argument holds for $\optdestset$. 
	\end{proof}
To apply the comparison principle (\Cref{th_soner}), we need to work with
a domain with a $\Cc^1$ boundary. To this end, we make the following assumption:
\begin{assumption}\label{assump_trajec}
	The set $\Gamma^*$ is nonempty and there exists $\bar{\eta}>0$ such that $\overline{\Oc}_{\bar{\eta}} \subset \Omega$.
	\end{assumption}
Then, for every $\mu >0$, and $\eta<\bar{\eta}$,
we select a function $\F^\mu_v : \Omegab \to \R$, that is $\Cc^d$ on $\Omega$, and that approximates $\F$ on  $\Oetab$, i.e.,
\begin{equation}\label{F_regu}
	\sup_{x\in\Oetab} | \F^\mu_v(x) -\F_v (x)|\leq  \mu \enspace . 
	\end{equation}
Let us also consider a domain $\Oeta^\mu$, deduced from $\F^\mu_v$ and defined as follows:
\begin{equation}\label{Oeta_regu}
	\Oeta^\mu = \{ x\in (\Omega \setminus (\sourceset \cup \destset)) \mid  \F^\mu_v (x) < \inf_{y \in \Omega} \{ \F_v (y) + \eta  \} \ \} \ . 
	\end{equation}
We notice that $ \mathcal{O}_{\eta - \mu}^\mu \subseteq \Oeta \subseteq \mathcal{O}_{\eta+\mu}^\mu$ for arbitrary small $\mu$. 
Moreover $\overline{\Oeta^\mu}\subset \Omega$, for  $\eta<\bar{\eta}$ and $\mu
\leq \bar{\eta}-\eta$. 
Then, $\Oeta^\mu$ can be compared with $\Oeta$, and it is 
the strict sublevel set of the $\Cc^d$ function. It can then be 
seen as a regularization of $\Oeta$. 
So, for almost all $\eta$ (small enough) and $\mu$ small enough,
 $(\partial \Oeta^\mu) \setminus (\destset\cup 
 \sourceset)$ is $\Cc^1$ (Corollary of Morse-Sard Theorem, see~\cite{morse1939behavior,sard1942measure}). 
We shall see that $\mathcal{O}_\eta^\mu$ is a smooth neighborhood of optimal trajectories, and we intend to reduce the state space of our optimal control problem from $\Omegab$ to the closure $\Oetabmu$ of $\Oetamu$. 

 Starting with the problem in direction ``to destination'', the reduction of the state space leads to a new set of controls:
\begin{equation}\label{controlseto}
	\mathcal{A}_{\eta,x} := \{\alpha \in \mathcal{A} \ | \ y_\alpha(x;s) \in \Oetabmu, \text{ for all }  s \geq 0 \} \enspace .
\end{equation}
Let $v\todest^{\eta}(x)$ denote the value function of the optimal control problem when the set of controls is $\mathcal{A}_{\eta,x}$. Consider a new state constrained HJ equation: $ SC(F,\Oetamu, (\partial \Oetamu) \cap (\partial \destset ))$, 
we have the following result:

\begin{proposition}[Corollary of~\Cref{th_soner}]
The value function $v\todest^{\eta}$ of the control problem in $\Oetabmu$ is the unique viscosity solution of $ SC(F,\Oetamu, (\partial \Oetamu) \cap (\partial \destset ))$. 
\end{proposition}
The same reduction works for the problem in the reverse direction. %
We denote $v^{\eta}\fromsource$ the value function of this problem with the set of controls be $\tilde{\A}_{\eta, x} :=\{ \tilde{\alpha} \in \A \mid \tilde{y}_{\tilde{\alpha}} (x;s) \in \Oetabmu, \text{ for all } s \geq 0 \}  $.
\begin{proposition}\label{lemmatra} 
	If $\ \Gamma^{*}$ is not empty, then $\Gamma^{*} \subseteq \Oetabmu $, and for all $x \in \Gamma^{*}$, we have: $v\fromsource (x) = v\fromsource^{\eta} (x), \ v\todest(x) = v\todest^{\eta} (x) $.
	
%
\end{proposition}
\begin{proof} $\Gamma^* \subseteq \Oetabmu$ is a straightforward
  consequence of \Cref{newboundary}. Then, we have $v\fromsource(x) \leq v\fromsource^\eta(x), v\todest (x)\leq v\todest^\eta(x)$ for all $x\in \Oetamu$, since $\Oetamu \subseteq \Omega$.
Then, we also have $v\fromsource(x) \geq v\fromsource^\eta(x), v\todest (x)\geq v\todest^\eta(x)$ for all $x\in \Gamma^*$, since there exists optimal trajectories from $x\in \Gamma^*$ staying in $\Gamma^*$ and $\Gamma^* \subseteq \Oetabmu$.
\end{proof}	

\subsection{$\delta$-optimal trajectories and the value function}
The above results express properties of exact optimal trajectories. We will also consider
approximate, $\delta-$optimal, trajectories.
\begin{definition}
	For every $x\in \Omega$, we say that $y_{\alpha^{\delta}}(x;\cdot):[0,\tau] \to \Omegab$ is a $\delta$-optimal trajectory with associated $\delta$-optimal control $\alpha^{\delta}:[0,\tau] \to S^1$ for the problem {\rm (\ref{dynmsys},\ref{disc-cost},\ref{value-disc})} if :
	\[
\ y_{\alpha^{\delta}}(x;\tau) \in \destset\quad\text{and}\quad \int_{0}^{\tau} e^{-t} dt  \leq v\todest(x) + \delta \enspace.
	\]
	We denote by $\Gamma^{\delta}_x$ the set of $\delta-$geodesic points starting from $x$, i.e., 
	\(\Gamma^{\delta}_x = \{ y_{\alpha^{\delta}}(x;t) \ | \ t\in[0,\tau], \ \alpha^\delta :[0,\tau] \to S^1\text{ $\delta$-optimal }\}\).
	We define analogously $\delta$-optimal trajectories for the problem in reverse direction, and denote by $\tilde{\Gamma}^{\delta}_x$ the set of $\delta$-geodesic points starting from $x$ in the reverse direction. 
\end{definition}

Following the same argument as in \Cref{geodesicpoints}, we obtain the following result:
\begin{proposition}\label{delta_geodesic}
	Let us denote
	\[ \optsourceset^\delta = \{ x \in \partial \sourceset \mid v\todest(x) \leq v^* +\delta \}, \ \optdestset^\delta = \{x \in \partial \destset \mid v\fromsource(x) \leq v^* + \delta \} \ , \]
	then we have: 
	\begin{equation}\label{delta_geoset}
\cup_{\delta'\in [0,\delta]}
	\cup_{x\in \optsourceset^{\delta-\delta'}} \{ \Gamma^{\delta'}_x \} = \cup_{\delta'\in [0,\delta]} \cup_{x\in \optdestset^{\delta-\delta'}} \{ \tilde{\Gamma}^{\delta'}_x \}.
	\end{equation}
\qed
	\end{proposition}
Let us denote the set in \eqref{delta_geoset} by $\Gamma^\delta$, and call it  \NEW{the set of $\delta-$geodesic points from $\sourceset$ to $\destset$}. In the following, we intend to deduce the relationship between $\Gamma^\delta$ and our $\eta-$neighborhood, $\Oeta$. Let us start with a property of the $\delta-$optimal trajectories.

\begin{lemma}\label{delta-programming}
	Let $y_{\alpha^{\delta}}(x;\cdot):[0,\tau_x^{\delta}]\to\Omegab$ be a $\delta-$optimal trajectory of Problem {\rm (\ref{dynmsys},\ref{disc-cost},\ref{value-disc})} with associated $\delta$-optimal control $\alpha^{\delta}$ and $\delta$-optimal time $\tau^{\delta}_x$. Assume $v\todest(x)<1$, i.e., the minimum time from $x$ to $\destset : \tau_x < + \infty$. For every $z = y_{\alpha^{\delta}}(x;t_z)$, let us define a control $\alpha': [0, \tau^\delta_x - t_z] \to S_1$ such that $ \alpha'(s) = \alpha^{\delta}(s+t_z), \forall s \in [0,\tau^\delta_x-t_z]$. Then, the associated trajectory starting in $z$ with control $\alpha'$, $y_{\alpha'}(z;\cdot):[0,\tau_x^{\delta}-t_z] \to \Omegab$, is at least $(e^{t_z}\delta)$-optimal for the problem {\rm (\ref{dynmsys},\ref{disc-cost},\ref{value-disc})} with initial state $z$. 
\end{lemma}
\begin{proof}
	By definition, we have $y_{\alpha^{\delta}}(x;\tau^{\delta}_x) \in \destset $,
and $\int_{0}^{\tau^{\delta}_x} e^{-t}dt \leq v\todest(x) + \delta. $ Then, considering the control $\alpha'$ defined above, we have $y_{\alpha'}(z;\tau_x^{\delta}-t_z) \in \destset$ and 
		\( e^{-t_z} \int_{0}^{\tau_x^{\delta}-t_z} e^{-s} ds \leq v\todest(x) - \int_{0}^{t_z}e^{-s}ds + \delta\). 
By the dynamic programming equation, we have
\( v\todest(x)\leq \int_{0}^{t_z}e^{-s}ds +e^{-t_z} v\todest(z)\),
 which implies
		\( \int_{0}^{\tau_x^{\delta}-t_z} e^{-s} ds 
\leq v\todest(z) + e^{t_z} \delta\). 
	We deduce the result from the definition of $(e^{t_z}\delta)$-optimal 
trajectories.
\end{proof}
\begin{remark}
	One can deduce the same result for the $\delta-$optimal trajectory of the problem in reverse direction. In fact, for the minimum time problem, our definition of the $\delta-$optimal trajectory implies $\tau^{\delta}_x - \tau^{*}_x \leq e^{\tau^{*}_x} \delta$, where $\tau^{\delta}_x$ and $\tau^{*}_x$ denote the $\delta-$optimal time and the true optimal time respectively.
\end{remark}

\begin{lemma}\label{delta-tra}
	For every $\eta > \delta >0$, we have $\Gamma^{\delta}\subseteq \Oetab $ .
\end{lemma}
\begin{proof}
	Let $y_{\alpha^{\delta'}}(\source;\cdot):[0,\tau] \to \Omegab$ denote  a $\delta'-$optimal trajectory for the problem  {\rm (\ref{dynmsys},\ref{disc-cost},\ref{value-disc})} with $\source \in \optsourceset^{\delta-\delta'}$ and $\delta'\leq \delta$.
Denote $\dest:=y_{\alpha^{\delta}}(\source;\tau) \in \destset$. 
It is sufficient to show that $y_{\alpha^{\delta'}}(\source;t_x) \in \Oetab$ for every $t_x \in [0,\tau]$. %
	
	For an arbitrary $t_x\in[0,\tau]$, let us denote $x:=y_{\alpha^{\delta'}}(\source;t_x)$. Let $\alpha': [0, \tau -t_x] \to S_1 $ be a control such that $ \alpha'(s) = \alpha^{\delta'}(s+t_x), \forall s \in [0,\tau-t_x]$. Then, the associated trajectory starting at $x$ with control $\alpha'$, $y_{\alpha'}(x;\cdot):[0,\tau-t_x] \to \Omegab$, satisfies $y_{\alpha'}(x;s) = y_{\alpha^{\delta'}}(\source; s+t_x)$, for every $s\in[0,\tau-t_x]$. %
By the definition of %
$T\todest$, and the fact that $y_{\alpha'}(x;\tau-t_x) = \dest\in \destset$,
we have $T\todest(x) \leq \tau-t_x$. 
Similarly, using simple change of variable $s = t_x - s'$, we have 
$T\fromsource(x) \leq t_x$. 
Thus we have  $T\fromsource(x) + T\todest(x) \geq (\tau - t_x) + t_x = \tau $.
By  an elementary computation, we have 
\begin{equation}v\fromsource(x) + v\todest(x) - v\fromsource(x) v\todest(x ) = 1 - e^{T\fromsource(x)+ T\todest(x)} \leq 1 - e^{-\tau} \leq v^* + \delta \ . \end{equation}
Using %
 \Cref{pro-tra}, and $\eta>\delta$,  we obtain that $x=y_{\alpha^{\delta}}(\source;t_x) \in \Oetab$ for all $t_x \in [0,\tau]$.
Since this is true for all $0\leq \delta'\leq \delta$, we obtain $\Gamma^{\delta}   \subseteq \Oetab$.
\end{proof}

\begin{lemma}\label{Oinclu-gamma}
	For every $\delta^{'}>0$, we have $\mathcal{O}_{\eta} \subset \Gamma^{\eta+\delta'} $. %
\end{lemma}
\begin{proof}
	Take a $x\in \mathcal{O}_{\eta}$, it is sufficient to show that there exists at least one $(\eta+\delta^{'})-$optimal trajectory from $\sourceset$ to $\destset$ 
	that passes through $x$.

	Suppose there exist an optimal trajectory from  $x$ to $\sourceset$, and an optimal trajectory from $x$ to $\destset$, then,  concatenating the reverse trajectory of the optimal trajectory from  $x$ to $\sourceset$, with the optimal trajectory from $x$ to $\destset$, we obtain an $\eta-$optimal trajectory from one point of $\sourceset$ to $\destset$ (by definition). %

	Otherwise, one can consider a $\frac{\delta'}{2}-$optimal trajectory from $x$ to $\sourceset$, $y_{\alpha_1}(x;\cdot): [0,\tau_1] \to \Omega,$ and a $\frac{\delta'}{2}-$optimal trajectory from $x$ to $\destset$, $\tilde{y}_{\tilde{\alpha}_2}(x;\cdot): [0,\tau_2] \to \Omega$.  %
	Then: 
	\[
	\begin{aligned}
	\int_{0}^{\tau_1+\tau_2} \!\!\!\!e^{-s}ds &= \int_{0}^{\tau_1} \!e^{-s} ds + e^{-\tau_1} \int_{0}^{\tau_2} \!e^{-s}ds 
        \leq v\fromsource(x) + \frac{\delta'}{2}+ (1-v\fromsource(x))(v\todest(x) + \frac{\delta'}{2})  \\
	& \leq v\fromsource(x)+v\todest(x)-v\fromsource(x)v\todest(x) + (\frac{\delta'}{2}+\frac{\delta'}{2}) \leq v^*+ (\eta+\delta') \enspace ,
 	\end{aligned}
	\]
	where the first inequality is deduced from $\tau_1 \geq T\fromsource(x)$. 
	Thus, concatenating those two trajectories, we obtain a  $(\eta+\delta')-$optimal trajectory from $y_{\alpha_1}(x;\tau_1)\in \sourceset$ 
to $\destset$.
\end{proof}

The above two lemmas entail that the sets of $\delta-$geodesic points $\Gamma^\delta$ and $\mathcal{O}_\eta$ constitute equivalent families of neighborhoods of the optimal trajectory, and, in particular, $\Oetab$ contains at least all $\delta-$optimal trajectories for every $\delta < \eta$. 
Moreover, the sets $\mathcal{O}_\eta$, and $\Oetamu$ are also equivalent families of neighborhoods of the optimal trajectory, since $ \mathcal{O}_{\eta - \mu}^\mu \subseteq \Oeta \subseteq \mathcal{O}_{\eta+\mu}^\mu$ for arbitrary small $\mu$. Based on these properties, we have the following result regarding the value functions. %
\begin{theorem}\label{valuefunction}
	For every $\delta < \eta$ and $x\in \Gamma^{\delta}$, we have 
	$v^{\eta}\fromsource(x) = v\fromsource(x)$ and $v^{\eta}\todest(x) = v\todest(x) $.
\end{theorem}
\begin{proof}
We have $v\fromsource(x) \leq v\fromsource^\eta(x), v\todest (x)\leq v\todest^\eta(x)$ for all $x\in \Oetamu$, since $\Oetamu \subseteq \Omega$.

Now, let $\delta < \eta$. We have $\Gamma^{\delta}\subset \Oetabmu$ for $\mu$ small enough, using \Cref{delta-tra}. Then, to show the reverse inequalities $v\fromsource(x) \geq v\fromsource^\eta(x), v\todest (x)\geq v\todest^\eta(x)$ for $x\in \Gamma^{\delta}$, it is sufficient to show that for all $\epsilon>0$,
 there exist 
$\epsilon$-optimal trajectories from $x\in \Gamma^\delta$ staying in $\Oetabmu$.

	Let $y_{\alpha^{\delta'}}(\source;\cdot): [0,\tau^{\delta'}] \to \Omegab $, be a $\delta'$-optimal trajectory from $\source$ to $\destset$,  with $\delta'\in [0,\delta]$, $\source \in \optsourceset^{\delta-\delta'}$ and $ y_{\alpha^{\delta'}}(\source;\tau^{\delta'}) =  \dest \in \destset$, and the associated $\delta'-$optimal control  $\alpha^{\delta'}$. Let $x = y_{\alpha^{\delta'}}(\source;t_x)$, for $t_x\in [0,\tau^{\delta'}]$.
Let $\epsilon>0$ and consider a $\epsilon$-optimal trajectory from $x$ to $\destset$ with time length $\tau'$. So we have
$v\todest(x) \leq (1-e^{-\tau'})\leq v\todest(x) + \epsilon$.
Replacing the trajectory $y_{\alpha^{\delta'}}(x;\cdot): [0,\tau^{\delta'}-t_x] \to \Omegab $ by the $\epsilon$-optimal trajectory from $x$ to $\destset$, we have 
$(1-e^{-\tau'})\leq v\todest(x) + \epsilon\leq 1-e^{-(\tau^{\delta'}-t_x)}+\epsilon$.
Then, we obtain a trajectory from $\source$ to $\destset$
with time $t_x+\tau'$ such that
$v\todest(\source) \leq (1-e^{-(t_x+\tau')})
\leq 1-e^{-\tau^{\delta'}}+e^{-t_x}\epsilon\leq v\todest(\source)+\delta+\epsilon$.
Then, this trajectory is in $\Gamma_{\source}^{\delta'+\epsilon}\subseteq \Gamma^{\delta+\epsilon}$. For $\epsilon$ small enough we have $\delta+\epsilon<\eta$, so 
 $\Gamma^{\delta+\epsilon}\subset \Oetabmu$,  for $\mu$ small enough. We deduce that the  $\epsilon$-optimal trajectory from $x$ to $\destset$ is included in $\Oetabmu$, which implies that 
$v\todest^\eta(x)\leq v\todest(x)+\epsilon$.
Since it is true for all $\epsilon$ small enough, we deduce 
$v\todest^\eta(x)\leq v\todest(x)$ and so the equality.

        By the same  arguments, we have $v^{\eta}\fromsource(x) = v\fromsource(x)$. 
\end{proof}

Based on the above results, if we are only interested to find $v^*$ and optimal trajectories between $\sourceset$ and $\destset$, we can focus on solving the reduced problem in the subdomain $\Oetamu$, i.e., solving the system $ SC(F,\Oetamu, (\partial \Oetamu) \cap (\partial \destset))$. %

\section{The Multi-level Fast-Marching Algorithm}\label{sec-algo}

We now introduce the multi-level fast-marching algorithm,
to solve the initial minimum time problem \eqref{problem}. 

To begin with, let us briefly recall the classical numerical schemes to solve the control problem {\rm (\ref{dynmsys},\ref{disc-cost},\ref{value-disc})} and associated system $SC(F,\Omega\setminus \destset , \partial\destset)$.  Assume a finite grid mesh $X$ discretizing $\Omega$ is given, and the approximate value function $V: X \to \R$ is obtained by solving a discretized equation of the form  
\begin{equation}\label{discrete_HJB}
	V(x) = \mathcal{U}(V)(x), \ \forall \ x \in X \ ,
\end{equation}
where $\mathcal{U}$ is an operator from $L^\infty(X)$ to itself, which is monotone and contracting for the sup-norm.
In~\eqref{discrete_HJB}, classical operators $\mathcal{U}$ are based on ``monotone'' finite difference or semi-Lagrangian discretizations of the system $SC(F,\Omega\setminus \destset , \partial\destset)$. 
Notice that \eqref{discrete_HJB} includes the boundary conditions. Since a solution $V$ to the discretized equation \eqref{discrete_HJB} is equivalently a fixed point of the operator $\mathcal{U}$, the existence and uniqueness of a solution $V$ to \eqref{discrete_HJB} follows from the contraction property of the operator $\mathcal{U}$.
The existence of such ``monotone'' schemes, as well as the convergence of the solution of discretized equation \eqref{discrete_HJB} to the solution of HJB equation, are often obtained for discounted problems under regularity assumptions such as~\Cref{assp1} and for a mesh step small enough, see for instance~\cite{kushner2001numerical} for finite difference discretization and~\cite{falcone2013semi} for semi-Lagrangian discretization. This includes the case of our 
problem {\rm (\ref{dynmsys},\ref{disc-cost},\ref{value-disc})}.
The usual numerical method to solve~\eqref{discrete_HJB} consists in 
updating $V$ successively in each node of $X$. The convergence of this method
to the fixed point of $\mathcal{U}$ follows from the contraction property.

\subsection{Classical Fast Marching Method}\label{sec-fast}
The fast marching method was first introduced by Sethian~\cite{sethian1996fast} and Tsitsiklis~\cite{tsitsiklis1995efficient} to deal with the front propagation problem, then extended to general static HJ equations. Its initial idea takes advantage of the property that the evolution of the domain encircled by the front is monotone non-decreasing, thus one is allowed to only focus on the computation around the front at each iteration.
Then, it is a single-pass method which is faster than standard iterative algorithms. Generally, it has computational complexity (number of arithmetic operations) in the order of $K_d M \log (M)$  in a $d$-dimensional grid with
$M$ points %
(see for instance~\cite{sethian1996fast,cristiani2007fast}). %
The constant $K_d$ is the maximal number of nodes of the
discrete neighborhoods that are considered,
so it depends on $d$ and satisfies $K_d \in [2d,L^d]$ where $L$ is the maximal 
diameter of discrete neighborhoods. For instance $K_d=3^d$ for 
a local semilagrangian discretization, whereas $K_d=2d$ for 
a first order finite difference discretization.

In fast marching method, $\mathcal{U}$ (as in~\eqref{discrete_HJB}) is also called the update operator. 
The fast marching algorithm
visits the nodes of $X$ in a special ordering and computes the approximate value function in just one iteration. The special ordering 
is such that the value function is monotone non-decreasing in the direction of propagation. This construction is done by dividing the nodes into three groups (see figure below): \Far, which contains the nodes that have not been searched yet; \Accepted, which contains the nodes
at which
the value function has been already computed and settled (by the monotone non-decreasing property of the front propagation, in the subsequent search, we do not need to update the value function of those nodes, see for instance~\cite{sethian2003ordered}); and \Narrow, which contains the nodes "around" the front (we only need to update the value function at these nodes).
\begin{figure}[H]
	\centering
	\includegraphics[width=0.5\textwidth]{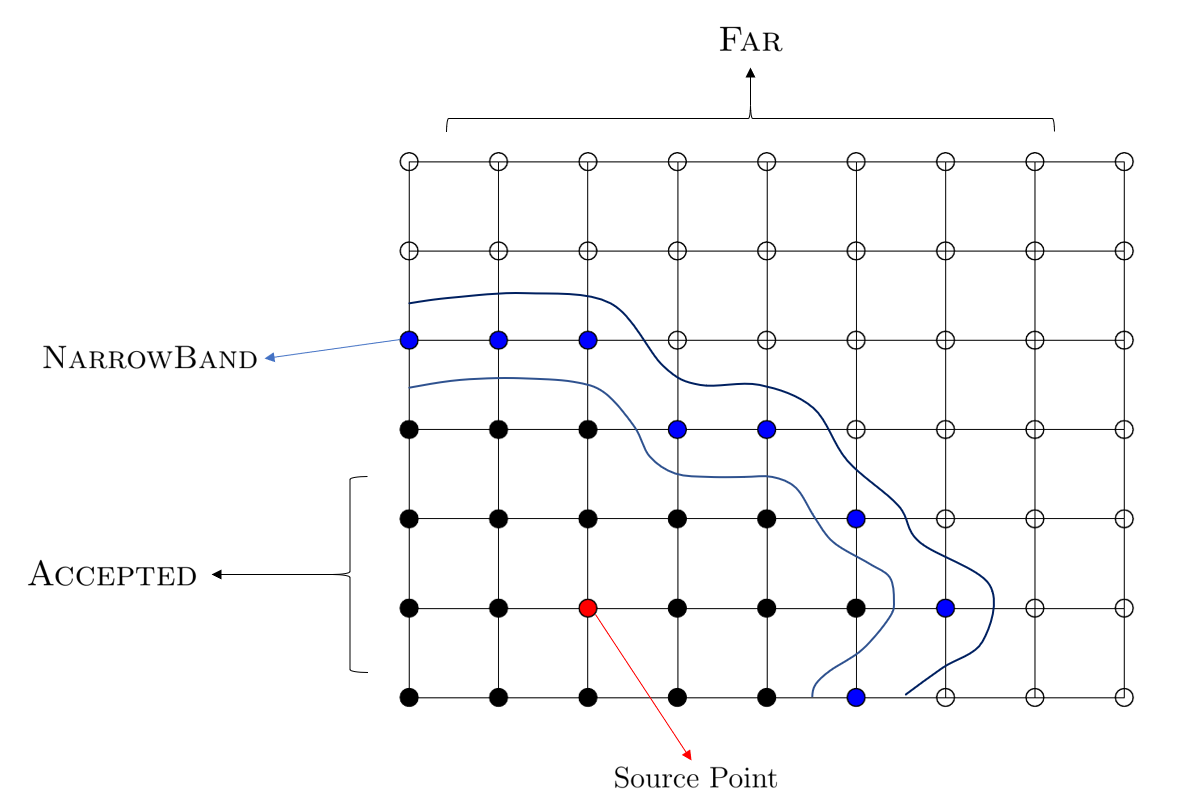}
	\label{fmfigure}
\end{figure}
At each step, the node in \Narrow  with the smallest value is added into
the set of \Accepted nodes, and then the \Narrow and the value function over \Narrow  are updated, using the value of the last accepted node. The computation is done by appying an update operator $\mathcal{U}$. 
Sufficient conditions on the update operator $\mathcal{U}$ 
for the convergence of the fast marching algorithm 
are %
that  $\mathcal{U}$ is not only monotone, but also causal (see for instance~\cite{sethian1996fast, cristiani2007fast}).

 A generic partial fast marching algorithm is given in \Cref{fmalgo} (compare with \cite{sethian1996fast,cristiani2007fast}). 
 We call it {\em partial} because the search stops when all the nodes of the ending set \End are accepted. 
Then, the approximate value function may only be computed in \End. %
The usual fast marching algorithm is obtained with \End equal to the mesh grid
$X$.
Moreover, for an eikonal equation, the starting set \Start plays
the role of the target (intersected with $X$).
If we only need to solve Problem \eqref{problem}, then we can apply
\Cref{fmalgo} with an update operator adapted to $SC(F,\Omega\setminus \destset , \partial\destset)$ with $F$ as in
\eqref{defF} and the sets 
\Start and \End equal to $\destset\cap X$ and $\sourceset\cap X$ respectively.
Similarly, we can apply \Cref{fmalgo} with an update operator adapted to
 the reverse HJ equation $SC(\tilde{F},\Omega\setminus \sourceset , \partial\sourceset)$ (with $\tilde{F}$ as in \Cref{sec_reverse}), which implies that the sets 
\Start and \End are equal to $\sourceset\cap X$ and $\destset\cap X$  respectively.

\begin{algorithm}[htbp]\small
	\caption{Partial Fast Marching Method. } 
	\label{fmalgo}
	\hspace*{0.02in} {{\bf Input:} Mesh grid $X$; Update operator $\mathcal{U}$. Two sets of nodes: \Start	and \End. }  \\
	\hspace*{0.02in} {{\bf Output:} Approximate value function $V$ and \Accepted set.}  \\
	\hspace*{0.02in} {{\bf Initialization:} Set $V(x) = +\infty,\forall x \in X$. Set all nodes as \Far. }
	\begin{algorithmic}[1]
		\State Add \Start to \Accepted, add all neighborhood nodes to \Narrow. 
		\State Compute the initial value $V(x)$ of the nodes in \Narrow.
		\While {(\Narrow \ is not empty and \End is not accepted)}
		\State Select $x^{*}$ having the minimum value $V(x^{*})$ among the \Narrow \ nodes.
		\State Move $x^{*}$ from \Narrow \ to \Accepted.  
		\For{All nodes $y$ not in \Accepted, such that $\mathcal{U}(V)(y)$ depends on $x^{*}$}
		\State $V(y) = \mathcal{U} (V)(y)$
		\If {y is not in \Narrow}
		\State Move $y$ from \Far \ to \Narrow.
		\EndIf
		\EndFor
		\EndWhile
	\end{algorithmic}
\end{algorithm}

\subsection{Two Level Fast Marching Method}
Our method combines coarse and fine grids discretizations, in order to obtain at a low cost, the value function on a subdomain of $\Omega$ around optimal trajectories. %
We start by describing our algorithm with only two levels of grid. %
\subsubsection{Computation in the Coarse Grid}\label{2lfm_coarse} We denote by  $X^H$ a coarse grid with constant mesh step $H$ on $\Omegab$, and by $x^H$ a node in this grid.
We perform the two following steps:
\begin{enumerate}
\item %
Do the partial fast marching search in the coarse grid $X^H$ in both froward and backward directions, 
to solve Problem \eqref{problem} as above.
\item %
  Select and store the \NEW{active} nodes (see~\Cref{OHcoarse}) 
  based on the two approximate value functions, as follows.
\end{enumerate}

The first step in coarse grid consists in applying \Cref{fmalgo} to each direction, that is 
a partial fast marching search, with mesh grid $X^{H}$, and
an update operator adapted to HJ equation $SC(F,\Omega\setminus \destset , \partial\destset)$ or $SC(\tilde{F},\Omega\setminus \sourceset , \partial\sourceset)$ and the appropriate sets \Start and \End. 
In particular, for a given parameter $\eta$, let us denote $\sourceset^{\eta} = \sourceset + B(0,{\eta})$, $\destset^{{\eta}} = \destset + B(0,{\eta})$. For the direction ``from source'', that is to solve the equation $SC(\tilde{F},\Omega\setminus \sourceset , \partial\sourceset)$, the set \Start is defined as  $\sourceset \cap X^H$, and the set \End is defined as $\destset^{{\eta}} \cap X^H$. For the other direction, the set \Start is defined as $\destset \cap X^H$ and the set \End is defined as  $\sourceset^{{\eta}} \cap X^H$.

This yields to the functions $V\fromsource^{H,1}$ and $V\todest^{H,1}$
that are numerical approximations of the value functions $v\fromsource$ and $v\todest$ on the sets of accepted nodes $A^H\fromsource$ and $A^H\todest$, respectively. 
Indeed,  $V\fromsource^{H,1}$
(resp.\ $V\todest^{H,1}$) is a function defined on all $X^H$, 
which coincides on the set of accepted nodes $A^H\fromsource$ (resp.\ $A^H\todest$) with the unique fixed point of the update operator, that is the solution 
$V\fromsource^{H}$ (resp.\ $V\todest^{H}$) of the 
discretized equation; elsewhere it may be lower bounded by $V\fromsource^{H}$
(resp.\ $V\todest^{H}$), or $+\infty$. 
Then, $\F_{V^{H,1}}$ is an approximation
of $\F_v$ on the set $A^H\fromsource\cap A^H\todest$
of accepted nodes for both directions, which should be an approximation of
the set of geodesic points. We thus construct an approximation of
$\Oeta$ as follows.
\begin{definition}\label{OHcoarse}
	For a given parameter $\etaH > 0$, we say that a node $x^H \in X^H$ is {\em active} if $x^H\in A^H\fromsource\cap A^H\todest$ and
	\begin{equation}\label{eqactive}
		\F_{V^H} (x^H) \leq \min_{y^H \in X^H} \F_{V^H}(y^H) + \etaH \ .
		\end{equation}
	We denote by $O^H_{\eta}$ the set of all active nodes for the parameter $\etaH$.
	\end{definition}

\subsubsection{Computation in the Fine Grid} Let us denote by $X^h$ a grid discretizing $\Omegab$ with a constant mesh step $h<H$. For the computation in the fine grid, we again have two steps:
\begin{enumerate}
\item Construct the fine grid, by keeping only the nodes of $X^h$ that are in a neighborhood   of the set of \NEW{active} nodes of the coarse grid.
	\item Do a fast marching search in forward direction in this fine grid. 
	\end{enumerate}

More precisely, we select the fine grid nodes as follows: %
\begin{equation}\label{constructfine}
	G^{h}_{\eta} = \{ x^{h} \in X^{h} \ | \ \exists  x^{H} \in O^{H}_{\eta} : \| x^{h} -x^{H} \|_\infty \leq \max (H-h, h) \} \enspace.
\end{equation}

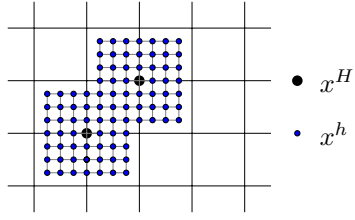
\begin{figure}[htbp]
	\centering
	\begin{tikzpicture}	[scale=0.7]
		\draw[step = 1, help lines, black] (0,0) grid (4,3);
		\fill (1,1) circle (.1);
		\fill (2,2) circle (.1);
		\draw[step = 0.25, help lines, gray] (0.25,0.25) grid (1.75,1.75);
		\draw[step = 0.25, help lines, gray] (1.25,1.75) grid (2.75,2.75);
		\draw[step = 0.25, help lines, gray] (1.75,1.25) grid (2.75,1.75);
		\draw[ help lines, gray] (1.75,1.25) -- (2.75,1.25) (1.25,1.75)--(1.25,2.75) ;
		\draw[ help lines, black] (-0.5,0) -- (0,0) (0,-0.5)--(0,0)(-0.5,1) -- (0,1) (-0.5,2) -- (0,2)(-0.5,3) -- (0,3)(1,-0.5)--(1,0)(2,-0.5)--(2,0)(3,-0.5)--(3,0)(4,-0.5)--(4,0)(0,3)--(0,3.5)(1,3)--(1,3.5)(2,3)--(2,3.5) (3,3)--(3,3.5)(4,3)--(4,3.5)(4,0)--(4.5,0)(4,1)--(4.5,1) (4,2)--(4.5,2) (4,3)--(4.5,3) ;
		\draw[fill = blue] (0.25,0.25) circle (.05)(0.25,0.5) circle (.05)(0.25,0.75) circle (.05)(0.25,1) circle (.05)(0.25,1.25) circle (.05)(0.25,1.5) circle (.05)(0.25,1.75) circle (.05)(0.5,0.25) circle (.05)(0.5,0.5) circle (.05)(0.5,0.75) circle (.05)(0.5,1) circle (.05)(0.5,1.25) circle (.05)(0.5,1.5) circle (.05)(0.5,1.75) circle (.05)(0.75,0.25) circle (.05)(0.75,0.5) circle (.05)(0.75,0.75) circle (.05)(0.75,1) circle (.05)(0.75,1.25) circle (.05)(0.75,1.5) circle (.05)(0.75,1.75) circle (.05)(1,0.25) circle (.05)(1,0.5) circle (.05)(1,0.75) circle (.05)(1,1.25) circle (.05)(1,.5) circle (.05)(1,1.75) circle (.05)(1,1.5) circle (.05)(1.25,0.25) circle (.05)(1.25,0.5) circle (.05)(1.25,0.75) circle (.05)(1.25,1) circle (.05)(1.25,1.25) circle (.05)(1.25,1.5) circle (.05)(1.25,1.75) circle (.05)(1.25,2) circle (.05)(1.25,2.25) circle (.05)(1.25,2.5) circle (.05)(1.25,2.75) circle (.05)(1.5,0.25) circle (.05)(1.5,0.5) circle (.05)(1.5,0.75) circle (.05)(1.5,1) circle (.05)(1.5,1.25) circle (.05)(1.5,1.5) circle (.05)(1.5,1.75) circle (.05)(1.5,2) circle (.05)(1.5,2.25) circle (.05)(1.5,2.5) circle (.05)(1.5,2.75) circle (.05)(1.75,0.25) circle (.05)(1.75,0.5) circle (.05)(1.75,0.75) circle (.05)(1.75,1) circle (.05)(1.75,1.25) circle (.05)(1.75,1.5) circle (.05)(1.75,1.75) circle (.05)(1.75,2) circle (.05)(1.75,2.25) circle (.05)(1.75,2.5) circle (.05)(1.75,2.75) circle (.05)(2,1.25) circle (.05)(2,1.5) circle (.05)(2,1.75) circle (.05)(2,2.25) circle (.05)(2,2.5) circle (.05)(2,2.75) circle (.05)(2.25,1.25) circle (.05)(2.25,1.5) circle (.05)(2.25,1.75) circle (.05)(2.25,2) circle (.05)(2.25,2.25) circle (.05)(2.25,2.5) circle (.05)(2.25,2.75) circle (.05)(2.5,1.25) circle (.05)(2.5,1.5) circle (.05)(2.5,1.75) circle (.05)(2.5,2) circle (.05)(2.5,2.25) circle (.05)(2.5,2.5) circle (.05)(2.5,2.75) circle (.05)(2.75,1.25) circle (.05)(2.75,1.5) circle (.05)(2.75,1.75) circle (.05)(2.75,2) circle (.05)(2.75,2.25) circle (.05)(2.75,2.5) circle (.05)(2.75,2.75) circle (.05);
		
		\fill (5,2) circle (.1);
		\node at (5.25,2)[right] {$x^{H}$};
		\draw[fill = blue] (5,1) circle (.05);
		\node at (5.25,1)[right] {$x^{h}$};
	\end{tikzpicture}
	\caption{Constructing the fine neighborhood $G^{h}_{\eta}$ given two active nodes $x^{H}$} %
\end{figure}
\begin{remark}
As we shall see in \Cref{sec-conv}, we may need to consider mesh steps $H$ and $h$ with $h$ close to $H$, in particular such that $h > \frac{1}{2} H$.
In this case, the bound in  \eqref{constructfine} is equal to $h$.
In general, this bound is more efficient numerically, although any bound
in $[H/2,H]$ would work theoretically.
\end{remark}
To solve the original minimum time problem, the computation will only be done in the selected fine grid nodes, which means that a full fast marching~\Cref{fmalgo} is applied in the restricted fine grid $G^{h}_{\eta} $,
with the update operator of one direction HJ equation (for instance with target set $\destset$).
We will denote by $V^{h,2}\todest$ the approximation of the value function $v\todest$ generated by the above 2-level algorithm on $G^h_\eta$. The complete algorithm is shown in \Cref{2LFM}.
\begin{algorithm}[hbtp]\small
	\caption{Two-Level Fast-Marching Method (2LFMM)} 
	\hspace*{0.02in} {{\bf Input:} Two grids  $X^h$ and $X^H$ with mesh steps $h<H$. The parameter $\eta_{H}>0$. }\\
\hspace*{0.02in} {{\bf Input:} 
Two update operators $\mathcal{U}\todest$ and $\mathcal{U}\fromsource$
adapted to both directions HJ equations.}
\hspace*{0.02in} {{\bf Input:} Target sets: $\sourceset,\destset$.}\\
\hspace*{0.02in} {{\bf Output:} The fine grid \Fine  and approximate value function $V\todest^{h,2}$ on \Fine.}
	\begin{algorithmic}[1]
\State Apply \Cref{fmalgo} with Input grid $X^H$, update operator $\mathcal{U}\todest$, \Start$=\destset\cap X^H$ and \End$=\sourceset^{\eta_H}\cap X^H$, and 
output $V\todest^{H,1}$ and $A^H\todest$.
\State Apply \Cref{fmalgo} with Input grid $X^H$, update operator $\mathcal{U}\fromsource$, \Start$=\sourceset\cap X^H$ and \End$=\destset^{\eta_H}\cap X^H$, and 
output $V\fromsource^{H,1}$ and $A^H\fromsource$.
		\For{Every node $x^{H}$ in $A\fromsource^H \cap A\todest^H$  }
		\If {$\F_{V^{H,1}(x)} \leq \min_{x^H \in X^H} \F_{V^{H,1}}(x^H) + \etaH$ }
		\State Set $x^{H}$ as \Active. 
		\EndIf
		\EndFor
		\State Set \Fine to emptyset.
		\For{Every node $x^{H}$ in the \Active set}
		\For{Every $x^{h} \in X^h$ satisfying $\|x^{h} - x^{H} \|_{\infty} \leq \max \{ H-h,h\} $}
		\If{$x^{h}$ does not exist in set \Fine}
		\State Add $x^{h}$ in the set \Fine. 
		\EndIf
		\EndFor 
		\EndFor
		\State Apply \Cref{fmalgo} with Input grid \Fine, update operator $\mathcal{U}\todest$, \Start$=\destset\cap \Fine$ and \End$=\sourceset\cap \Fine$, and 
output $V\todest^{h,2}$. 
	\end{algorithmic}	\label{2LFM}
\end{algorithm}

\subsubsection{Convergence of~\Cref{2LFM}}\label{section-correct2}
In order to show the correctness of~\Cref{2LFM}, we first show that the computation in the fine grid is equivalent to the approximation of the value
function of a new optimal control problem, with a restricted state space.

For this purpose, one shall first construct a continuous extension 
$O_{\eta}^{H,I}$ of  $O^{H}_{\eta}$ and $G^{h}_{\eta}$. 
Let us extend the approximate value function $V\fromsource^{H,1}$ and $V\todest^{H,1}$ from the nodes of $ X^{H}$ to the whole domain $\Omegab$ by a linear interpolation, and denote them by $V\fromsource^{H.I}, V\todest^{H,I}$ respectively. %
Then,  $O_{\eta}^{H,I}$ is defined as follows
\begin{equation}\label{OHIcoarse}
    O^{H,I}_{\eta} = \{ x \in (\Omega \setminus (\sourceset \cup \destset) ) \mid \F_{V^{H,I}}(x) < \min_{x^H \in X^H} \F_{V^{H,1}}(x^H) + \etaH   \} \enspace.
\end{equation}
Note however that another method may consists in constructing 
the region $O_{\eta}^{H,I}$ in the same way as $G^{h}_{\eta}$,
but without the constraint $x\in X^h$.
Nevertheless, 
$O^{H,I}_{\eta}$ can be thought of as
a continuous version of the set $O^{H}_{\eta}$ of active nodes in coarse grid.
We shall relate $O^{H,I}_{\eta}$ 
to the domain $\Oeta$ defined in \eqref{eqo}  -- the notation ``$I$''  stands for ``interpolation''. 
 We then consider the continuous optimal control problem {\rm (\ref{dynmsys},\ref{disc-cost},\ref{value-disc})} with new state space  $\overline{O^{H,I}_\eta}$, and the following new set of controls which is adapted to the new state constraint:
\begin{equation} 
	\A_{\etaH,x} = \{\alpha \in \mathcal{A} \ | \ y_{\alpha}(x; s) \in  \overline{O_{\eta}^{H,I}}, \text{ for all } s\geq0 \ \} \enspace.
\end{equation}
Denote by $v\todest^{\etaH,I}$ the value function of this new state constrained problem. By \Cref{th_soner}, it is the unique solution of the new state constrained HJ equation $SC(F,$ $O_{\eta}^{H,I},(\partial O_{\eta}^{H,I}) \cap (\partial \destset ))$. In our two level fast marching algorithm, we indeed use the grid $G^h_\eta$ to discretize $ \overline{O_{\eta}^{H,I}}$, then $V\todest^{h,2}$ is an approximation of $v\todest^{\etaH,I}$. %
Then, if $\overline{O_{\eta}^{H,I}}$ is big enough to contain the true optimal trajectories, by the results of \Cref{sec-cotra},  $v\todest^{\etaH,I}$ coincides with $v\todest$ on the optimal trajectories. Then, $V\todest^{h,2}$ is an 
approximation of $v\todest$ on  optimal trajectories.
To present our results, we shall make the following assumptions for the error estimates of the classical numerical scheme.
\begin{assumption}\label{assup_error}
For  every open bounded domain $\Omega'\subset \Omega$ of $\R^d$
containing $\sourceset$ and $\destset$,
we denote by $V^{H,\Omega'}\fromsource$ and $V^{H,\Omega'}\todest$ the 
value functions computed using fast-marching method
on (that is restricted to) the grid $X^H\cap \Omega'$.
Then, there exist constants $C\fromsource, C\todest, C_{\Oc}>0$ and $0 < \gamma \leq 1$, such that, for all $x\in X^H\cap\Omega$, we have 
\begin{equation}\label{eqassup_error_upper}
v\fromsource(x)- V^{H,\Omega}\fromsource(x) 
 \leq C\fromsource H^\gamma , \quad v\todest(x)- V^{H,\Omega}\todest(x)  \leq C\todest H^\gamma \ ,
	\end{equation} 
 and for every $\Omega'$ as above, and every $\eta>0$ and $H>0$
satisfying $\Omega'\supset \overline{\Oc}_{\eta}$, 
with $\eta\geq C_{\Oc} H^{\gamma}$, and  for all $x\in X^H\cap \overline{\Oc}_{\eta/2}$, we have 
\begin{equation}\label{eqassup_error}
V^{H,\Omega'}\fromsource(x) - v\fromsource(x)  \leq C\fromsource H^\gamma , \quad
V^{H,\Omega'}\todest(x) - v\todest(x) \leq C\todest H^\gamma .
	\end{equation} 
\end{assumption}

\begin{remark}\label{remark_error}
The parameter $\gamma$ measures the convergence rate of fast marching method or of the discretization scheme.
In unconstrained cases, the usual finite differences or semi-lagrangian schemes have an error estimate in $O(H^\gamma)$ with $\gamma = \frac{1}{2}$, which entails that \eqref{eqassup_error_upper} and \eqref{eqassup_error} hold for all 
grid points $x$. Moreover, $\gamma$ is equal to $1$ under semiconcavity assumptions (see for instance~\cite{falcone2013semi}).
For state constrained problems, similar estimates have been
established for instance in~\cite{camilli1996approximation,de2023optimal}.
They are stated under further regularity properties on
$\partial \Omega$ and particular discretization grids.
The regularity is needed to derive the bound
\eqref{eqassup_error} for all $x\in X^H\cap \Omega$, when $\Omega'=\Omega$.
Hence, showing that \eqref{eqassup_error} holds for all $x\in X^H\cap \Omega'$,
with constants  $C\fromsource$ and $C\todest$ independent of the set $\Omega'$,
seems to be generally difficult.
However, in the above assumption, we only require \eqref{eqassup_error} to
hold for grid points belonging to $\overline{\Oc}_{\eta/2}$, which are thus
at a distance of at least $\eta/(2 L_v)$ from the boundary of $\Omega'$.
This assumption arises naturally when analysing the convergence and 
complexity of our algorithm. Since the experimental complexity is consistent with 
the theoretical bounds it is tempting to conjecture that the above assumption
holds for the minimum time problem.
This question and more generally the question of the dependence of the
error estimates in the characteristics of the domain 
seems to be of independent interest.
	\end{remark}
	
In the following result, we denote (as for $h=H$) by $V^h\todest$ 
the solution of the discretization of
the HJ equation $SC(F,\Omega,\partial \destset)$ (associated to
Problem \eqref{problem}) on the grid $X^h$, or equivalently the
unique fixed point of $\mathcal{U}\todest$, 
that is the output of \Cref{fmalgo} with input grid $X^h$, update operator $\mathcal{U}\todest$, \Start$=\destset\cap X^h$ and \End$=X^h$.

\begin{theorem}[Convergence of the Two-Level Fast-Marching Method]\label{lemmaOH} $ $
	\begin{enumerate}
		\item\label{lemmaOH-i} 
		Denote $C_\gamma := C\fromsource+C\todest$ and $L_v := L_{v\fromsource}+L_{v\todest}$, where $L_{v\fromsource}$ and $L_{v\todest}$ are the Lipschitz constants of $v\fromsource$ and  $v\todest$, respectively.
		Then, there exists a constant $C_\eta> 0$ depending on 
		$C_\gamma $, $ C_{\Oc}$ and $L_v$, such that for every $\delta'>\delta > 0$, 
		for every 
		$\etaH \geq C_\eta H^{\gamma} + \delta'$, 
		$\overline{O_{\eta}^{H,I}}$ contains the set $\overline{\Od_{\delta'}}\supset \Gamma^{\delta}$, that is the set of $\delta-$geodesic points for the continuous problem {\rm (\ref{dynmsys},\ref{disc-cost},\ref{value-disc})}. In particular, taking $\etaH \geq 2 C_\eta H^\gamma$ and $\delta' = \frac{\etaH}{2}$, we have 
\[ 
		\Gamma^{\delta}\subset	\overline{\Od_{\frac{\etaH}{2}}} \subset \overline{O_{\eta}^{H,I}} \ .
\] 
		\item \label{lemmaOH-ii}
		Taking $\etaH$  and $\delta'=\etaH/2$ as in \eqref{lemmaOH-i}, then for every  $\delta < \etaH/2$ and $x \in X^h \cap \Gamma^{\delta}$, 
\[ 
			| V^{h,2}\todest (x) - v\todest(x) | \leq C\todest h^{\gamma}  \ .
\] 
		Thus,  $V^{h,2}\todest(x)$ converges towards $v\todest(x)$ as $h \to 0$.
	\end{enumerate}
\end{theorem}

\begin{proof}
  Let us prove Point \eqref{lemmaOH-i}.
  By~\Cref{assump_trajec} and \Cref{assup_error}, for any $x^H\in X^H$,
  we have:
  \begin{equation}\label{1ineq}
    |\F_{V^{H}}(x^H) - \F_v(x^H) | \leq C_\gamma H^{\gamma} \; ,
  \end{equation}
  where $C_\gamma = C\fromsource+C\todest$.
  Moreover, using the Lipschitz continuity of $v\fromsource$ and $ v\todest$,
  and denoting $L_v = L_{v\fromsource}+L_{v\todest}$,   
  we obtain, for any $x\in\Omegab$ and $x^H\in X^H$ such that $\|x^H-x\|\leq H$: 
	\begin{equation}
          \label{2ineq}  |\F_{V^{H}}(x^H) - \F_v(x) | \leq
          C_\gamma H^{\gamma}+L_v H \enspace .\end{equation}
        Applying \eqref{1ineq} and $X^H\subset \Omegab$, 
        we get 
        \begin{equation}\label{lowerboundH}
          \min_{x^H \in X^H} \F_{V^H}(x^H) + C_\gamma H^{\gamma} \geq \min_{x \in \Omega} \F_{v}(x)\enspace .\end{equation}
Assume that $x$ is in a $d$-dimensional polytope with vertices in 
$X^H$ and that $V^{H,I}\todest$ and $V^{H,I}\fromsource$ are linear or affine 
on this polytope. One can show, using that both functions $V^{H,I}\todest$ and $V^{H,I}\fromsource$ take their values in
$[0,1]$, that the maximum (and minimum) of  $\F_{V^{H,I}}$ on this polytope is attained on
the vertices, so on points of $X^H$ (although in some cases the function
is concave).
Using  this property with  \eqref{2ineq}, we obtain
\[  \F_{v}(x) \geq \F_{V^{H,I}}(x) - (C_\gamma H^{\gamma}+L_v H) \enspace .\]
Let us assume now that $x\in \overline{\Omega \setminus (\sourceset \cup \destset)}\setminus \overline{O^{H,I}_{\eta}}$.
Then, we deduce from the previous inequalities:
\[ 
		\begin{aligned}
			\F_{v}(x) 
			&\geq  \min_{x^H \in X^H} \F_{V^H}(x^H) +\etaH - (C_\gamma H^{\gamma} + L_v H) \\
			& \geq \min_{x \in \Omega} \F_v(x) + \etaH - (2C_\gamma H^{\gamma} + L_v H) \ .
		\end{aligned}
\] %
	Thus, if we take $\etaH \geq  2C_\gamma H^{\gamma} + L_v H + \delta'$, we obtain $x \notin \Od_{\delta'}$. 	This shows that $\Od_{\delta'}\subset  \overline{O^{H,I}_{\eta}}$ and thus $\overline{\Od_{\delta'}}\subset  \overline{O^{H,I}_{\eta}}$. 
Since $\gamma \leq 1$, we can take 
	$\etaH \geq C_\eta H^{\gamma} +\delta' $, with an appropriate constant $C_\eta$ and also $\etaH \geq 2 C_\eta H^{\gamma} $ with $\delta'= \etaH/2$,
 that is the result of Point \eqref{lemmaOH-i}. 
	
	As for Point~\eqref{lemmaOH-ii}, taking $\etaH$ and $\delta'=\etaH/2$ as in~\eqref{lemmaOH-i}, we first notice that for every $x \in\overline{O^{H,I}_{\eta}}\cap X^h$, we have
	\begin{equation}\label{prove_ii_1}
		| V^{h,2}\todest (x) - v^{\eta_H,I}\todest (x) | \leq C\todest h^\gamma \ ,
	\end{equation}
	by~\Cref{assup_error}. 
	Since $ \overline{\Od_{\frac{\etaH}{2}}} \subset \overline{O^{H,I}_{\eta}} \subset \Omegab$, we have for every $x \in \Od_{\frac{\etaH}{2}}$: 
	\begin{equation}\label{prove_ii_2}
		v\todest(x) \leq v^{\eta_H,I}\todest(x) \leq v^{\frac{\etaH}{2}}\todest (x) \ .
	\end{equation}
	By~\Cref{valuefunction}, we have that for every $\delta < \etaH/2$ and $x \in \Gamma^{\delta}\subset  \overline{\Od_{\frac{\etaH}{2}}}$, $v\todest(x) = v^{\frac{\etaH}{2}}\todest (x)$. Thus we get an equality in~\eqref{prove_ii_2}. Replacing $ v^{\etaH,I}\todest$ by $v\todest$ in~\eqref{prove_ii_1}, we obtain the result of Point~\eqref{lemmaOH-ii}.
\end{proof}

\subsection{Multi-level Fast Marching Method}
The computation in the two-level coarse/fine grid can be extended to the multi-level case. We construct finer and finer grids, considering the fine grid
of the previous step as the coarse grid of the current step, and defining the next
fine grid by selecting the actives nodes of this coarse grid.
\subsubsection{Computation in Multi-level Grids} Consider a $N$-level family of grids with successive mesh steps $H_{1} \geq H_{2} \geq \dots \geq H_{N-1} \geq H_{N} = h,$ denoted $X^{H_i}$, for $i=1,\ldots, N$. Given a family of real positive parameters $\{\eta_1, \eta_2, \dots ,\eta_{N-1} \}$, the computation works as follows:

\textbf{\textit{Level-$1$:}} In the first level, the computations are the same as in the coarse grid of the two-level method (\Cref{2lfm_coarse}), with mesh step $H$ equal to $H_1$ and active nodes selected using the parameter $\eta$ equal to $\eta_1$. At the end of level-$1$, we get a set of active nodes: 
$O^{H_1}_{\eta_1}$.

\textbf{\textit{Level-$l$ with $1<l<N$:}} In level-$l$, we already know the set of active nodes in level-$(l-1)$, denoted $O^{H_{l-1}}_{\eta_{l-1}}$. 
We first construct the "fine grid" set of level-$l$ as in \eqref{constructfine}, that is:
\begin{equation}
	G^{H_{l}}_{\eta_{l-1}} \!=\! \{ x^{H_{l}} \!\in X^{H_{l}}\!\mid\! \exists \ x^{H_{l-1}} \!\in O^{H_{l-1}}_{\eta_{l-1}} : \| x^{H_{l-1}} - x^{H_{l}} \|_\infty \leq \max (H_{l-1}-H_{l},H_{l})  \} .
\end{equation}
Then, we perform the fast marching in both directions in the grid $G^{H_{l}}_{\eta_{l-1}}$. 
This leads to the approximations $V\fromsource^{H_{l},l} $,  $V\todest ^{H_{l},l} $
and $\F_{V^{H_l,l}}$ of $v\fromsource$,  $v\todest$ and
$\F_v$ on grid $G^{H_{l}}_{\eta_{l-1}}$. 
 We then select the active nodes in level-$l$, by using the parameter $\eta_l$, as follows:
\begin{equation}
	O^{H_{l}}_{\eta_{l}} = \{ x^{H_{l}} \in G^{H_{l}}_{\eta_{l-1}} \ | \ \F_{V^{H_l,l}} (x^{H_l}) \leq \min_{ x^{H_{l}} \in G^{H_{l}}_{\eta_{l-1}} } \F_{V^{H_l,l}} (x^{H_l}) + \eta_l \ \} \enspace. 
\end{equation}

\textbf{\textit{Level-N:}} In the last level, we only construct the final fine grid:
\begin{equation}
	G^{h}_{\eta_{N-1}} = \{ x^{h} \!\in X^{h} \!\mid\! \exists \ x^{H_{N-1}} \!\in O^{H_{N-1}}_{\eta_{N-1}} : \| x^{h} - x^{H_{N-1}} \|_\infty \leq \max (H_{N-1}-h, h )  \} .
\end{equation}
Then, we only do the fast marching search in one direction in the grid $G^h_{\eta_{N-1}}$ and obtain the approximation $V\todest ^{h,N} $ of $v\todest$
on grid $G^{h}_{\eta_{N-1}}$. 

This is detailed in \Cref{MLFM}. Some admissible grids generated  by our algorithm are shown in \Cref{sketchmlfh}.
\begin{figure}[H]
	\centering
	\subfigure[Level-0]{
		\includegraphics[width=0.18\textwidth]{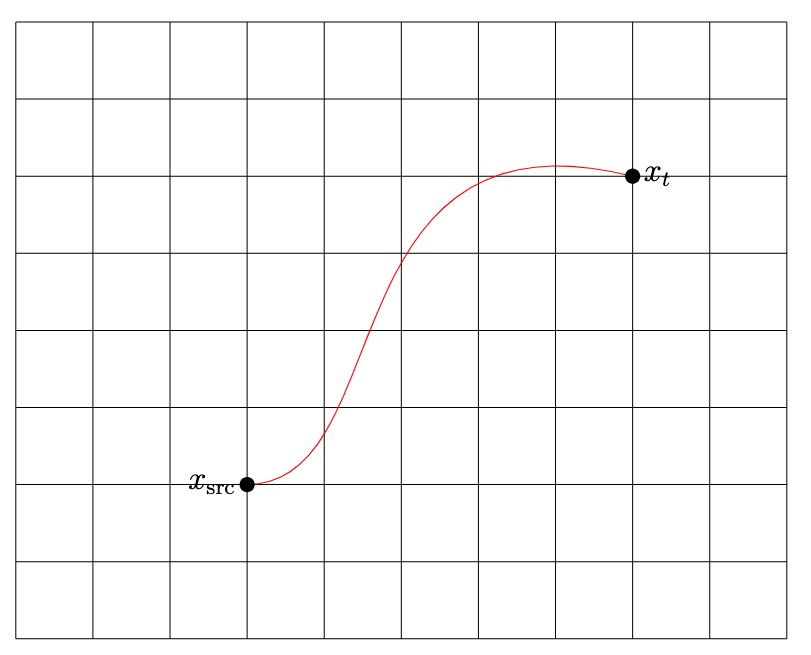}}
	\subfigure[Active Nodes]{
		\includegraphics[width=0.185\textwidth]{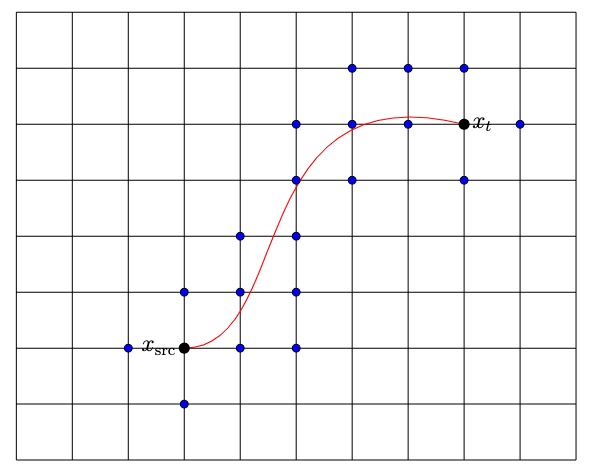}}
	\subfigure[Fine grid]{
		\includegraphics[width=0.155\textwidth]{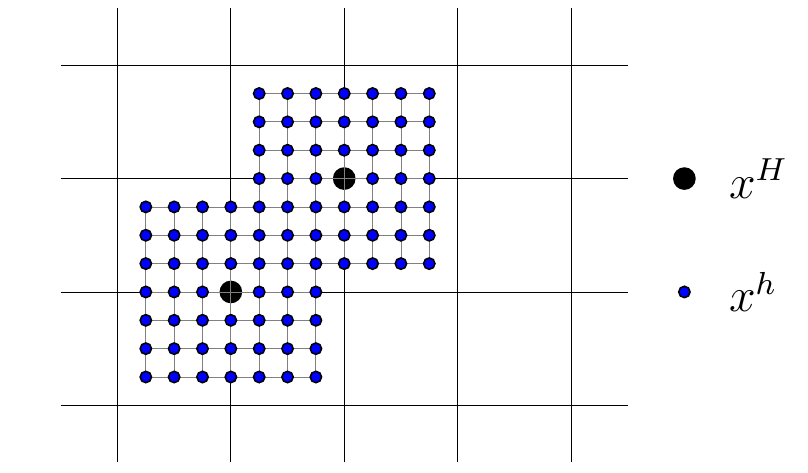}}
	\subfigure[Level-1]{
		\includegraphics[width=0.185\textwidth]{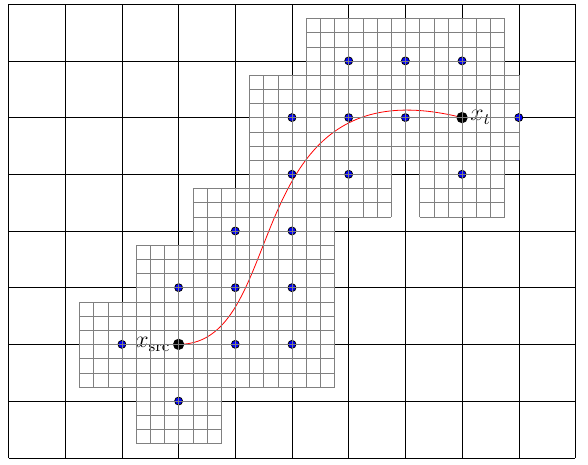}}
	\dots
	\subfigure[Level-2]{
		\includegraphics[width=0.15\textwidth]{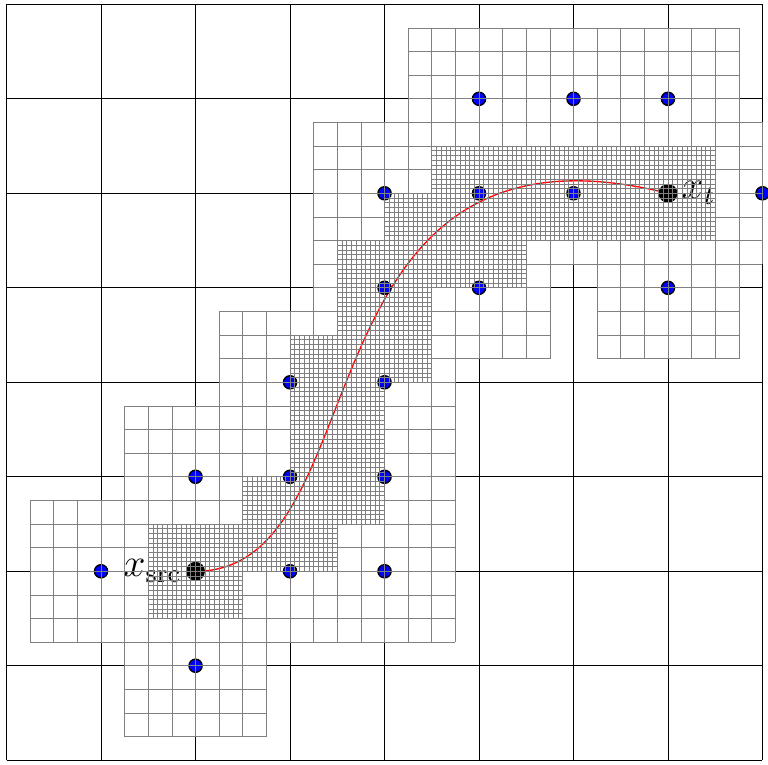}}
	
	\caption{Sketch of MLFMM.}
	\label{sketchmlfh}
\end{figure}
\begin{algorithm}[htbp]\small
	\caption{Multi-Level Fast-Marching Method (MLFMM)} 
		\hspace*{0.02in} {{\bf Input:} The mesh steps, grids, and selection parameters: $H_{l},X^{H_l},\eta_{l}$, for $l \in \{1,2,\dots,N\}$. }\\
	\hspace*{0.02in} {{\bf Input:} 
		Update operators $\mathcal{U}\todest$ and $\mathcal{U}\fromsource$
		adapted to both directions HJ equations and different levels.} \\
\hspace*{0.02in} {{\bf Input:} Target sets: $\sourceset,\destset$.}\\
\hspace*{0.02in} {{\bf Output:} The final fine grid \Fine  and approximate value function $V\todest^{h,N}$ on \Fine.}
	\begin{algorithmic}[1]
		\State 	Set \Coarsegrid to $X^{H_1}$.
		\For {$l = 1 \text{ to } N-1$}
		\State Do the partial fast marching search in \Coarsegrid \ in both directions. 
		\State Select the \Active nodes from the \Accepted nodes using $\eta_{l}$.
		\State Select the \Fine nodes based on the \Active nodes, and mesh step $H_{l+1}$.
		\State Let \Fine in current level be the new \Coarsegrid.
	\EndFor
	\State Do the partial fast marching search in only one direction in \Fine.
	\end{algorithmic}	\label{MLFM}
\end{algorithm}

In \Cref{MLFM}, 
Line-3 of  \Cref{MLFM} corresponds to lines-1 and 2 in \Cref{2LFM}, 
line-4 of  \Cref{MLFM}  corresponds to lines-3 to 7 in \Cref{2LFM}, 
line-5  of  \Cref{MLFM} corresponds to line-8 to line-15 in \Cref{2LFM}. 
\subsubsection{Convergence of~\Cref{MLFM}}
In each level-$l$ with $l<N$, 
we have the approximate value functions $V^{H_{l},l}\fromsource$ and $V^{H_{l},l}\todest$ of $v\fromsource$ and $v\todest$ on the grid of level $l$. 
Then, we can apply the same constructions as in \Cref{section-correct2}
for the coarse grid $X^H$. This leads to the following continuous version
of the set of active nodes in level-$l$:
\[ O_{\eta_{l}}^{H_{l},I} = \{ x \in (\Omega \setminus (\sourceset \cup \destset) ) \mid 
\F_{V^{H_l,I}}(x) < \min_{x^H \in  G^{H_{l}}_{\eta_{l-1}} }\F_{V^{H_l,l}}(x^H) + \eta_l   \} \enspace.
\]
We also have sets $\mathcal{A}_{\eta_{l},x}$ of controls adapted 
to the optimal control problems of the form
{\rm (\ref{dynmsys},\ref{disc-cost},\ref{value-disc})}
with state space equal to the set $O_{\eta_{l}}^{H_{l},I}$,
and the corresponding value function $v^{\eta_{l},I}\todest$,
which,
by \Cref{th_soner}, is the unique solution of the new state constrained HJ equation $SC(F,O_{\eta_{l}}^{H_{l},I}, (\partial O_{\eta_{l}}^{H_{l},I}) \cap (\partial \destset))$. 
We can also consider the HJ equations in the direction "from source",
$SC(\tilde{F},O_{\eta_{l}}^{H_{l},I},  (\partial O_{\eta_{l}}^{H_{l},I}) \cap (\partial \sourceset))$, and the corresponding value function $v^{\eta_{l},I}\fromsource$. 
We obtain the following convergence result for \Cref{MLFM}.
\begin{theorem}[Convergence of the Multi-level Fast-Marching Method] 
	\label{theo-corr-multi} $ $
	\begin{enumerate}
		\item\label{proofmlfm_i}
There exists constants $C_\eta> 0$ and $\kappa>0$ such that, if for every $l \in \{ 1,\dots, N-1 \}$, $\eta_l = C_\eta (H_l)^{\gamma}$ and $H_l/H_{l+1}\geq \kappa$, then for all $\delta<\frac{\eta_l}{2}$, the set $\overline{O^{H_l,I}_{\eta_l}}$ contains $\overline{\Od_{\frac{\eta_l}{2}}}\supset \Gamma^\delta$, that is the set of $\delta$-
geodesic points for the continuous minimum time problem{\rm (\ref{dynmsys},\ref{disc-cost},\ref{value-disc})}. 
		\item\label{proofmlfm_ii} %
 Taking $\eta_l$ as proposed in \eqref{proofmlfm_i}, then 
		for every $l \in \{1,2, \dots, N-1 \}$, 
$\delta <\frac{\eta_{l}}{2}$, and 
$x \in X^{H_{l+1}}\cap \Gamma^\delta$, we have 
\[ 
			| V^{H_{l+1},l+1}\todest (x) - v\todest(x) | \leq C\todest (H_{l+1})^\gamma, \ | V^{H_{l+1},l+1}\fromsource (x) - v\fromsource(x) | \leq C\fromsource (H_{l+1})^\gamma \ . 
\] 
		Thus, $ V^{h,N}\todest(x)$ converges towards $v\todest(x)$ as $h \to 0$.
	\end{enumerate}
\end{theorem}
\begin{proof}
  When $l=1$, Points \eqref{proofmlfm_i} and \eqref{proofmlfm_ii} of \Cref{theo-corr-multi} follow from the corresponding   points in~\Cref{lemmaOH}: $\overline{\Od_{ \frac{\eta_1}{2} }} \subset \overline{O^{H_1,I}_{\eta_1}}$, and for every $x \in H^{H_2} \cap \Gamma^\delta$ with $\delta< {\frac{\eta_1}{2}}$:
	\begin{equation}\label{prove-l-1}
	| V^{H_2,2}\todest(x) - v\todest(x) | \leq C\todest H_2^\gamma, \quad 
	|V^{H_2,2}\fromsource (x) - v\fromsource(x)| \leq C\fromsource H_2^\gamma  \ .
	\end{equation} 
	Assume now that Points \eqref{proofmlfm_i} and \eqref{proofmlfm_ii} of~\Cref{theo-corr-multi} hold for $l=k-1$, with an arbitrary $k\geq 2$. 
	We first prove Point \eqref{proofmlfm_i} for $l=k$, that is  $\Od_{\frac{\eta_{k}}{2}} \subset \overline{O^{H_{k},I}_{\eta_{k}}}$. 
	Let us denote as before by $V^{H_{k}}\fromsource$, $V^{H_{k}}\todest$ the solutions in the grid $X^{H_{k}}$ of the discretized equation in two directions, respectively. From~\Cref{assup_error}, we obtain \eqref{1ineq},
        which in the case $H=H_k$ writes 
	\begin{equation}
		\begin{aligned}
		  &\sup_{x \in X^{H_{k}}} \|\F_{V^{H_{k}}} (x) - \F_v(x) \| \leq C_\gamma (H_{k})^\gamma  \enspace .
		%
		\end{aligned}
	\end{equation}
	Moreover, we notice that, for every $x \in G^{H_{k}}_{\eta_{k-1}}\subset X^{H_k}$, we have
		$V^{H_{k},k}\fromsource (x) \geq V^{H_{k}}\fromsource (x)$ and 
		$V^{H_{k},k}\todest(x) \geq V^{H_{k}}\todest(x) $. 
        Hence, 
        \begin{equation}
	  \F_{V^{H_{k},k}} (x) \geq \F_{V^{H_{k}}} (x)
          \geq \F_v(x)-  C_\gamma (H_{k})^\gamma            \ ,
	\end{equation}
        which is similar to \eqref{lowerboundH} for $H=H_k$.        
	Consider a point $x \in \overline{ \Omega \setminus (\sourceset \cup \destset) } \setminus \overline{O^{H_{k},I}_{\eta_{k}}}$. 
        By definition of $\overline{O^{H_{k},I}_{\eta_{k}}}$, we have 
	\begin{equation}\label{prove-l-2}
		\begin{aligned}
			\F_{V^{H_{k},I}}(x) &\geq  \min_{x^{H_{k}} \in G^{H_{k}}_{\eta_{k-1}}} \F_{V^{H_{k},k}} (x^{H_{k}}) + \eta_{k} \\\ 
			& \geq \min_{x^{H_{k}} \in X^{H_{k}}} \F_v(x^{H_{k}}) + \eta_{k} - C_\gamma (H_{k})^\gamma  \\
			& \geq \min_{x \in \Omega} \F_v(x) + \eta_{k} - C_\gamma (H_{k})^\gamma  \ .
		\end{aligned}
	\end{equation} 
	Assume also that $x \in \Od_{\frac{\eta_{k}}{2}}$. Let us denote
	\begin{equation} 
		B^{H_{k}} (x)= \{ x^{H_{k}}  \in X^{H_{k}} \mid  \| x - x^{H_{k}} \|_\infty \leq H_{k} \  \} \  , 
	\end{equation}
	and assume further that $B^{H_{k}} (x) \subset \Gamma^\delta$
        for some $\delta<\frac{\eta_{k-1}}{2}$.
        Using \Cref{Oinclu-gamma}, and the Lipschitz continuity of $\F_v$,
        we get this property as soon as
        $\eta_k+2 L_v H_k < \eta_{k-1}$.
        Since we assumed that Point~\eqref{proofmlfm_i} holds for $l=k-1$,
        and $ B^{H_{k}} (x)\subset X^{H_k}\cap\Gamma^\delta$,
        we have
        \[ \|\F_{V^{H_{k},k}} (x^{H_{k}}) - \F_v(x^{H_{k}}) \| \leq C_\gamma (H_{k})^\gamma \enspace ,\]
        for all $x^{H_{k}}\in B^{H_{k}} (x)$.
        Then, by the same arguments as in the proof of \Cref{lemmaOH},
        using~\eqref{prove-l-2}, we obtain
        	\begin{equation}\label{prove-l-3}
		\begin{aligned}
			\F_{V^{H_{k},I}} (x) &\leq \max_{x^{H_{k}}\in B^{H_{k}}(x)}  \F_{V^{H_{k},k}}(x^{H_{k}}) 
		 \leq \F_v(x) + C_\gamma (H_{k})^\gamma + L_v H_{k} \\
			& \leq \min_{y \in \Omega} \F_v(y) + \frac{\eta_{k}}{2} + C_\gamma (H_{k})^\gamma + L_v H_{k} \\
                        & \leq 	\F_{V^{H_{k},I}} (x) - \frac{\eta_{k}}{2} + (2 C_\gamma (H_{k})^\gamma + L_v H_{k}) \\
		\end{aligned}
	        \end{equation}
                If $\eta_k$ satisfies also
                $\eta_{k} \geq 4C_\gamma (H_{k})^\gamma + 2L_v H_{k}$,
                we get a contradiction.
	        Since $\gamma \leq 1$, this condition is satisfied as soon as
                $\eta_{k} \geq C_\eta (H_{k})^\gamma$ with some appropriate constants $C_\eta$. If $\eta_k$ also satisfies $\eta_{k} \leq C'_\eta (H_{k})^\gamma$ for some constant $C'_\eta>C_\eta$, we get that the above condition
                $\eta_k+2 L_v H_k< \eta_{k-1}$ holds as soon as
                $H_k\leq 1$ and $\frac{C'_\eta+2 L_v}{C_\eta} \leq (\frac{H_{k-1}}{H_{k}})^\gamma$.
                Under these conditions, 
                we deduce that for all $\delta<\frac{\eta_{k}}{2}$,
                we have  $\Gamma^{\delta}\subset \Od_{\frac{\eta_{k}}{2}} \subset \overline{O^{H_{k},I}_{\eta_{k}}}$.

	        By the same argument as in the proof of Point~\eqref{lemmaOH-ii} of~\Cref{lemmaOH}, we deduce Point~\eqref{proofmlfm_ii} of~\Cref{theo-corr-multi} for $l=k$ from Point~\eqref{proofmlfm_ii} of~\Cref{theo-corr-multi} for the same $l=k$.
	The result of~\Cref{theo-corr-multi} follows from the induction on $l$.
\end{proof}

\section{Computational Complexity}\label{sec-conv}
We now analyze the space and time complexity of our algorithm, and determine the optimal parameters to tune the algorithm.
We always make the following assumption:
\begin{assumption}\label{asspf_2}
The domain $\Omegab$ is convex and 
	there exist constants $\lowerboundf,\upperboundf$ such that:	
	$        0<\lowerboundf\leq f(x,\alpha)\leq \upperboundf <\infty, \text{ for all } x\in \Omega \text{ and } \alpha \in \mathcal{A}$.
\end{assumption}
Consider first the two-level case, and suppose we want to get a numerical approximation of the value of Problem \eqref{problem} with an error bounded by some given $\varepsilon>0$, and a minimal total complexity.
Three parameters should be fixed before the computation: 
\begin{enumerate}
\item The mesh step of the fine grid $h$;
	\item The mesh step of the coarse grid $H$;
	\item The parameter $\etaH$ to select the active nodes in coarse grid.
\end{enumerate} 
The parameter $h$ needs to be small enough so that $C\todest h^\gamma\leq \varepsilon$. 
The parameter $\etaH$ is used to ensure that the subdomain $O^{H,I}_{\eta}$ does contain the true optimal trajectories, so it should be large enough
as a function of $H$, see \Cref{lemmaOH}, but we also
want it to be as small as possible to reduce the complexity.
Using the optimal values of $h$ and $\etaH$, the total
complexity becomes a function of $H$, when $\varepsilon$ (or $h$) is fixed. 
Then, one needs to choose the parameter $H$ to minimize
this total complexity. 

In order to estimate this total complexity, we shall make the following assumption on the neighborhood of optimal trajectories:
\begin{assumption}\label{distance_beta}
The set of geodesic points $\Gamma^*$ consists of a finite number of optimal paths between $\sourceset$ and $\destset$. Moreover, 
	for every $x\in \mathcal{O}_{\eta}$, there exists $x^*\in \Gamma^*$ such that :
	$\|x - x^*\| \leq C_\beta \eta^\beta$, 
	where $0<\beta \leq 1$, and $C_\beta$ is a positive constant.
	\end{assumption} 

\begin{remark}\label{rk-beta}
  The above constant $\beta$ depends on the geometry of the level sets of the value function.
In typical situations in which the value
function is smooth with a nondegenerate Hessian in the neighborhood
of an optimal trajectory, one has $\beta=1/2$.
However, in general situations, it can take all the possible values in $(0,1]$.
Indeed, consider for instance the minimum time between two points in the $x-$axis, $(-1,0)$ and $(1,0)$, 
 in a 2-dimensional space with $f(x,\alpha)= 1/\|\alpha\|_p$, where
$\|\cdot\|_p$ is the $L^p$ norm, with $p\in [1,\infty]$. This is equivalent to a problem with a dynamics independent of state and a direction chosen from the unit $L^p$ sphere instead of the $L^2$ sphere.
Then, one can show that for $p\neq \infty$, the straight line is the unique optimal path, and that the value function satisfies \Cref{distance_beta} with $\beta=1/p\in (0,1]$. However, if $p=\infty$, the number of optimal paths is infinite, so
\Cref{distance_beta} is not satisfied.
  \end{remark}
Let us denote by $D$ the maximum Euclidean distance between the points in $\sourceset$ and $\destset$, i.e.\ $D = \sup \{ \|x-y \| \mid x\in \sourceset, y \in \destset  \}$, and by $D_\Omega$  the diameter of $\Omegab$, i.e.\ 
$D = \sup \{ \|x-y \| \mid x,y\in \Omegab  \}$.
Recall that for any positive functions $f,g:\R^p\to(0,+\infty)$ of $p$ real parameters,
the notation $g(x)=\widetilde{O}(f(x))$ means
$g(x)=O(f(x)(\log(f(x)))^q)$ for some integer $q$,
that is $|g(x)|\leq C f(x)|\log(f(x))|^q$ for some constant $C>0$.
We have the folowing estimate of the space complexity.
\begin{proposition}\label{complexity-2level}
Assume that $H\leq D$, and that $\etaH$ satisfies the condition of \Cref{lemmaOH} and $\etaH\leq D$.
There exists a constant $C>0$  depending on $D_\Omega$, 
$D$, $\upperboundf$ and $\lowerboundf$, 
$\beta$, $\gamma$, $C_\beta$, $C_\gamma$ and $L_v$ (see \Cref{lemmaOH}),
such that
       the space complexity $\spacecom(H,h)$ of the two level fast marching algorithm with coarse grid mesh step $H$, fine grid mesh step $h$ and parameter $\eta_H$  is as follows:
	\begin{equation} \label{e-complexity}
		\spacecom(H,h) = \widetilde{O}\Big( C^d \Big( \frac{1}{H^d} + \frac{(\eta_H)^{\beta(d-1)}}{h^d} \Big)\Big) \enspace.
	\end{equation}
\end{proposition}
\begin{proof}
  Up to a multiplicative factor, in the order of $d$ (so which enters in the 
$\widetilde{O}$ part) the space complexity is equal to the total
  number of nodes of the coarse and fine grids.

  We first show that up to a multiplicative factor, the first term
  in~\eqref{e-complexity},  $(\frac{C}{H})^{d}$, is the number of accepted nodes in the coarse-grid.
  To do so, we exploit 
  the monotone property of the fast-marching update operator. Recall that in the coarse grid, we incorporate dynamically new nodes by partial fast-marching, starting from $\sourceset\cap X^H$ (resp. $\destset\cap X^H$) until $\destset\cap X^H$ (resp. $\sourceset\cap X^H$) is accepted. 
  
  Let us first consider the algorithm starting from $\sourceset$. In this step, let us denote $\dest^f$ the last accepted node in $\destset$, then we have for all the nodes $x^{H}\in A\fromsource^H$ that have been accepted, $V^H\fromsource(x^{H}) \leq V^H\fromsource(\dest^f)$. Then, by~\Cref{assup_error}, we obtain  $v\fromsource(x^{H}) \leq v\fromsource(\dest^f)+2 C\fromsource H^\gamma \leq v\fromsource(\dest^f)+2 C\fromsource D^\gamma$,
 which gives the following inclusion, when $D$  and $T\fromsource(\dest^f)$ are
small enough: 
  \begin{align*}
A\fromsource^H
&\subseteq \{ x\in \Omegab \mid v\fromsource(x) \leq v\fromsource(\dest^f) +2C\fromsource H^\gamma\} \\
&\subseteq \{ x\in \Omegab \mid T\fromsource(x) \leq T\fromsource(\dest^f) -\log(1-2C\fromsource D^\gamma e^{T\fromsource(\dest^f)}) \} \ . 
\end{align*}

Let $x\in \Omegab$,  recall that $T\fromsource(x)$ is the minimum time traveling from $x$ to $\sourceset$, then we have for some $\source^i \in \sourceset$,
  \begin{equation}\label{leftside}
  	\| x - \source^i \| \leq \int_{0}^{T\fromsource(x)} \|\dot{x}(t)\| dt \leq \upperboundf T\fromsource(x) \ .
  	\end{equation}
  Moreover, we have for some $\source^j \in \sourceset$,
  \begin{equation}\label{rightside}
  	T\fromsource(\dest^f) \leq \frac{\| \dest^f - \source^j \|}{\lowerboundf} \leq \frac{D}{\lowerboundf}\ ,
  	\end{equation}
  since we can take a control $\alpha$ proportional to $\dest^f - \source^j$, so that the trajectory given by \eqref{dynmsys_t} follows the straight line from $\source^j$ to $\dest^f$ (with variable speed). Combine \eqref{leftside} and \eqref{rightside}, we have for the set of accepted nodes:
 \begin{align*}
A\fromsource^H &\subseteq \{x \mid  \| x- \source^i\| \leq ({\upperboundf D}/{\lowerboundf})
-\upperboundf  \log(1-2C\fromsource D^\gamma e^{{\upperboundf D}/{\lowerboundf}})  , \;  \text{for some} \ \source^i \in \sourceset \} \ .
\end{align*}
  Thus, all the nodes we visit are included in a $d$-dimensional ball with radius $R$, where $R$ is a constant depending on $D$, ${\upperboundf}$, ${\lowerboundf}$, $C\fromsource$ and $\gamma$, when $D$ is small enough.
Otherwise, since $\Omegab$ has a diameter equal to $D_\Omega$, one can take $R=D_\Omega$.
Then, the total number of nodes that are accepted in the coarse grid  is bounded by $(\frac{C}{H})^{d}$, in which $C/R$ is a positive constant in the order of 
$(\upsilon_d)^{1/d}$, where $\upsilon_d$ is the volume of unit ball in $\R^d$,
and satisfying $C/R\leq 2$, so we can take $C/R=2$. 
  The same result can be obtained for the search starting from $\destset$.

  We now show that, still up to a multiplicative factor,
  the second term in~\eqref{e-complexity},
  $D\frac{(\eta_H)^{\beta(d-1)}}{h^d}$, is the number of nodes of the fine-grid. 
  Consider a node $x^{h} \in G^{h}_{\eta}$. By definition, there exists $x^H\in O^H_\eta$ such that $\|x^h-x^H\|\leq H$. 
	 Denote $C_\gamma = C\fromsource+C\todest$, and $L_v = L_{v\fromsource}+L_{v\todest}$ (the sum of the Lipschitz constants of $v\todest$ and $v\fromsource$).
Then, using similar arguments as in the proof of \Cref{lemmaOH}, we obtain:
\begin{align*}
\F_{v}(x^h) &\leq \F_v(x^H)+L_v \|x^h-x^H\| 
 \leq \F_{V^H}(x^H)+C_\gamma H^\gamma+ L_v H\\
&\leq  \min_{y^H\in X^H} \F_v(y^H)+\eta_H+ C_\gamma H^\gamma+L_v H
\leq  \min_{y\in \Omega } \F_v(y)+\eta_H+ C_\gamma H^\gamma+2 L_v H\enspace .
\end{align*}

This entails that $x^h \in \mathcal{O}_{\eta_H + \epsH }$,
with $\epsH=C_\gamma H^\gamma+2 L_v H$. 
Since $\gamma \leq 1$, so $\epsH$ is of order $H^\gamma$, 
and $\etaH\geq C_\eta H^\gamma$ by assumption,
then using \Cref{distance_beta}, 
we deduce that for some positive constant $C'$, depending on 
$\beta,\; \gamma$, $C_\beta$, $C_\gamma$, and $L_v$, 
 there exists $x^*\in \Gamma^*$
such that 
\begin{equation}\label{nodedistance}
	\|x^{h}-x^* \| \leq C' (\etaH)^\beta \enspace.
\end{equation}
We assume now that $\Gamma^*$ consists of a single optimal path
from $\source\in\sourceset$ to $\dest\in\destset$.
Indeed, the proof for the case of a finite number of paths is similar
and leads to a constant factor in the complexity, which 
is equal to the number of paths.
Up to a change of variables, we get a parametrization of $\Gamma^*$ as the image of a one-to-one map $\phi: t\in [0,D_{\Gamma}]\mapsto \phi(t)\in \Gamma^*$, 
with unit  speed $\|\phi'(t)\|=1$, where $D_\Gamma$ is a positive constant.
Let us denote in the following by $d_{\Gamma^*}(x,y)$ the distance between two points $x,y \in \Gamma^*$ along this path,
that is $d_{\Gamma^*}(x,y)=|t-s|$ if $x=\phi(t)$ and $y=\phi(s)$, with $t,s\in  [0,D_{\Gamma}]$. Since the speed of $\phi$ is one,
we have $\|x-y\|\leq d_{\Gamma^*}(x,y)$, and so $D\leq D_\Gamma$.
Moreover, by the same arguments as 
above, we have $D_\Gamma\leq \upperboundf D/\lowerboundf$.

Let us divide  $\Gamma^*$, taking equidistant points $x_0, x_1,x_2,...,x_N,x_{N+1} \in \Gamma^*$ between $x_0=\source$ and $x_{N+1}= \dest$, with $N = \lfloor 
D_\Gamma 
 /( C'(\eta_H)^\beta )\rfloor$, so that 
$d_{\Gamma^*}(x_{k},x_{k+1}) = D_\Gamma/ (N+1) \leq C' (\eta_H)^\beta\; \forall k \in \{ 0, 1,2,...,N \}$. 
Set $\Gamma^*_{\mathrm{dis}}:= \{ \source, x_1,x_2,...,x_N, \dest \}$.
Then, by \eqref{nodedistance}, we have for every $x^h \in G^h_\eta$, there exists a point $x \in \Gamma^*_{\mathrm{dis}}$ such that
\( \|x^h - x  \| \leq \frac{3}{2} C' (\eta_H)^\beta \). 
Let us denote $B^d(x,r)$ the $d$-dimensional open ball with center $x$ and radius $r$ (for the Euclidean norm).
Taking, $\Delta := \frac{3}{2}C' (\etaH)^\beta  + \frac{h}{2}$, we deduce:
\[ \bigcup_{x \in G^h_\eta}B^d(x,\frac{h}{2}) \subset \bigcup_{x\in \Gamma^*_{\mathrm{dis}}} B^d(x,\Delta)  \ .\]
Moreover, since the mesh step of $G^h_\eta$ is $h$, all balls centered in $x\in G^h_\eta$ with radius $\frac{h}{2}$ are disjoint,
which entails:
\[\operatorname{Vol} \Big(\bigcup_{x \in O^h_\eta}B^d(x,\frac{h}{2})\Big) = |G^h_\eta| (\frac{h}{2})^d \upsilon_d \ , \]
where $\upsilon_d$ denotes the volume of the unit ball in dimension $d$, and $|G^h_\eta|$ denotes the cardinality of $G^h_\eta$, which is also the number of nodes in the fine grid. Thus, we have:
\[
|G^h_\eta| = \frac{\operatorname{Vol} \Big(\bigcup_{x \in O^h_\eta}B^d(x,\frac{h}{2})\Big)}{(\frac{h}{2})^d \upsilon_d} \leq \frac{\operatorname{Vol}\Big( \bigcup_{x\in \Gamma^*_{\mathrm{dis}}} B^d(x,\Delta) \Big)}{ (\frac{h}{2})^d \upsilon_d} \leq |\Gamma^*_{\mathrm{dis}}| 2^d \frac{\Delta^d}{h^d} \enspace .
\]
Since $\etaH\geq C_\eta H^\gamma$, $h\leq H$ and $\beta,\gamma\leq 1$, we get that  
$\Delta\leq C'' (\etaH)^\beta$ for some constant $C''$ depending on $C_\eta$, 
$C'$, $\beta$ and $\gamma$.
Then, 
\( |G^h_\eta|  \leq D' \myfrac{(2C''(\etaH)^{\beta})^{d-1} }{h^d}\), 
where $D'$ depends on $D_\Gamma$, $D$ and $C'$.
This leads to the bound of the proposition.
\end{proof} 
 \begin{remark}
For the fast marching method with semi-lagrangian scheme, the computational complexity $\computcom$  satisfies  $\computcom=\widetilde{O}(3^d \spacecom)$.
Then, the same holds for the two-level or multi-level fast marching 
methods. %
In particular for the  two-level fast marching method
$\computcom$ has same estimation as $\spacecom$ in \Cref{complexity-2level}.
 	\end{remark}

The same analysis as in two level case also works for the $N-$level case, for which we have the following result:
\begin{proposition}\label{prop-multi-complexity}
Assume that $H_1\leq D$, and that, 
for $l=1,\ldots, N-1$, $\eta_l$ satisfies the condition of \Cref{theo-corr-multi}, and $\eta_l\leq D$.
Then, 	the total computational complexity $\computcom (H_{1},H_{2},\dots,H_{N})$ of the multi-level fast marching algorithm with $N$-levels, with grid mesh steps $H_{1} \geq H_{2} \geq ,\dots, \geq H_{N-1} \geq H_{N} = h$ is 
	\begin{equation}\label{com_ml}
		\computcom (\{H_l\}_{1 \leq l \leq N}) = \widetilde{O}\Big( C^d \Big( \frac{1}{(H_{1})^d} + \frac{(\eta_1)^{\beta(d-1)}}{(H_2)^d}+ \frac{(\eta_2)^{\beta(d-1)}}{(H_3)^d}+ \dots+ \frac{(\eta_{N-1})^{\beta(d-1)}}{h^d}\Big)\Big)\enspace ,
	\end{equation}
with $C$ as in \Cref{complexity-2level}.
Moreover, the space complexity has a similar formula.\hfill \qed
\end{proposition}

Minimizing the expressions in \Cref{complexity-2level}
 and \Cref{prop-multi-complexity}, 
we obtain the following computational complexity.

\begin{theorem}\label{complexity_1}
Assume $d\geq 2$, and let $\nu:=\gamma\beta (1-\frac{1}{d})<1$.
Let $\varepsilon>0$, and choose $h = (C_\gamma^{-1} \varepsilon)^{\frac{1}{\gamma}}$.
Then, there exist some constant $C_m$ depending on the same parameters as 
in \Cref{complexity-2level}, such that,
in order to obtain an error bound on the value of Problem \eqref{problem} 
less or equal to $\varepsilon$ small enough,
one can use one of the following methods:
	\begin{enumerate}
		\item\label{complexity_i} The two-level  fast marching method with $ \eta_H =C_\eta H^ \gamma$, and 
$H = h^{\frac{1}{\nu+1}}$.
In this case, the total computational complexity is $\computcom(H,h) = \widetilde{O}((C_m)^d (\frac{1}{\varepsilon})^{\frac{d}{ \gamma (\nu+1)}} )$. 
		\item\label{complexity_ii}
The $N-$level fast marching method
with $\eta_l = C_\eta H^\gamma_l$ and $H_l = h^{ \frac{1 - \nu^l }{ 1 - \nu^N } }$, for $l=1,\dots,N-1$. In this case,  the total computational complexity is  $\widetilde{O}(N (C_m)^d (\frac{1}{\varepsilon})^{ \frac{1 - \nu}{ 1 - \nu^N }\frac{d}{\gamma} }  )$. %
		\item\label{complexity_iii} The $N-$level fast marching method 
with  $N=\lfloor \frac{1}{\gamma}  \log(\frac{1}{\varepsilon}) \rfloor$, 
and  $\eta_l = C_\eta H_l^\gamma$ and $H_l =h^{\frac{l}{N}}$,
for $l=1,\dots,N-1$. 
 Then, the total computational complexity reduces to $\widetilde{O}((C_m)^d (\frac{1}{\varepsilon})^{(1-\nu)\frac{d}{\gamma}})=\widetilde{O}((C_m)^d (\frac{1}{\varepsilon})^{\frac{1+(d-1)(1-\gamma \beta)}{\gamma}})$.
When $\gamma=\beta=1$, it reduces to $\widetilde{O}((C_m)^d \frac{1}{\varepsilon})$.
		\end{enumerate}
	\end{theorem}
\begin{proof}
  For~\eqref{complexity_i}, 
by~\Cref{assup_error} 
together with \Cref{lemmaOH},
which applies since $ \eta_H = C_\eta H^ \gamma$,
 we get that the error on the value obtained by the
 two-level  fast marching method 
is less or equal to $C_\gamma h^\gamma=\varepsilon$.
Note that in order to apply 
\Cref{lemmaOH},
$\eta_H$ needs to satisfy $ \eta_H \geq C_\eta H^ \gamma$. 
Then, to get a minimal computational complexity, one need
to take $ \eta_H =C_\eta H^ \gamma$ as in the theorem.
We obtain the following total computational complexity: 
  \begin{equation}\label{com_2l_1}
	\computcom (H,h) = \widetilde{O}( (C')^{d} (H^{-d} + {h^{-d}} H^{ \gamma \beta(d-1)} ))\enspace,
\end{equation}
for some new constant $C'=C\max(1, C_\eta)^{\beta}$.
When $h$ is fixed, this is a function of $H$ which gets its minimum value for
\begin{equation}\label{minim_H_1}
	H = 
C_1  h^{\frac{d}{(\gamma \beta+1)d-\gamma \beta}} \quad \text{with}\quad C_1=\big(\frac{d}{ \gamma \beta(d-1)}\big)^{\frac{1}{(\gamma \beta + 1)d- \gamma \beta}}\enspace,
	\end{equation}
since it is decreasing before this point and then increasing.
Then, the minimal computational complexity bound is obtained by 
substituting the value of  $H$ of \eqref{minim_H_1} in \eqref{com_2l_1}.
The formula of $H$ in \eqref{minim_H_1}  is as in~\eqref{complexity_i}, up to the multiplicative factor $C_1>0$. One can show that $1\leq C_1\leq C_2$ with $C_2$ depending only on $\gamma\beta$.
Hence, substituting this value of $H$ instead of the one of \eqref{minim_H_1}
in  \eqref{com_2l_1},
gives the same bound up to the multiplicative factor $C_1^d$.
So in both cases, using
$h = (C_\gamma^{-1} \varepsilon)^{\frac{1}{\gamma}}$,
we obtain a computational complexity bound as in~\eqref{complexity_i},
for some constant $C_m$ depending on the same parameters as 
in \Cref{complexity-2level}. 
  
  For~\eqref{complexity_ii}, 
using this time~\Cref{assup_error} 
together with \Cref{theo-corr-multi}, 
and that $\eta_l = C_l H_l^{\gamma}$, we get that the error on the value obtained by the multi-level fast marching method 
is less or equal to $C_\gamma h^\gamma=\varepsilon$.
As in~\eqref{complexity_i}, we 
apply the formula of \Cref{prop-multi-complexity}
 with $\eta_l = C_l H_l^{\gamma}$, which gives 
  \begin{equation}\label{com_ml_2}
  	\begin{aligned}
  \computcom&(\{H_l\}_{1\leq l \leq N})  \\
  & = \widetilde{O}\Big((C')^d ((H_1)^{-d}  +(H_1)^{\gamma \beta(d-1)} (H_2)^{-d} + \dots + (H_{N-1})^{ \gamma \beta(d-1)} h^{-d}  \Big) \ ,
  \end{aligned}
  \end{equation}
for the same constant $C'$ as above.
This is again a function of $H_1, H_2,\dots,H_{N-1}$ when $h$ is fixed. We deduce the optimal mesh steps $\{ H_1, H_2,\dots,H_{N-1} \}$ by taking the minimum of $\computcom(\{H_l\}_{1\leq l \leq N})$ with respect to $H_1, H_2,\dots,H_{N-1}$,
and then simplifying the formula by eliminating the constants. 
We can indeed proceed by induction on $N$, and use iterative formula 
similar to \eqref{minim_H_1}. 
We then obtain the formula for $H_l$ as in~\eqref{complexity_ii}.  Substituting these values of the $H_l$ into \eqref{com_ml_2}, 
for a general $d\geq 2$, we obtain the following bound on the 
total computational complexity 
\begin{equation}\label{comp-N}
	\computcom(\{H_l\}_{1\leq l \leq N}) =\widetilde{O}(N (C')^d (\frac{1}{h})^{ \frac{1-\nu}{1-\nu^N} d})
	\enspace .\end{equation}
Now using $h = (C_\gamma^{-1} \varepsilon)^{\frac{1}{\gamma}}$, 
we obtain~\eqref{complexity_ii}.

For \eqref{complexity_iii}, 
let us  first do as if $\nu =1$.
 In that case, passing to the limit when $\nu$ goes
to $1$ in previous formula, we obtain the new formula for the $H_l$ given in
 \eqref{complexity_iii}.
We also obtain a new formula for the total complexity in \eqref{comp-N}
of the form $\widetilde{O}(N (C')^d (\frac{1}{h})^{ \frac{d}{N}})$.
The minimum of this formula with respect to $N$ is obtained for
$N = d\log(\frac{1}{h})$ which with
 $h = (C_\gamma^{-1} \varepsilon)^{\frac{1}{\gamma}}$, 
leads to a formula of $N$ in the order of  the one of~\eqref{complexity_iii}.
Let us now substitute the values of $H_l$ into the complexity formula~\eqref{com_ml_2}, we obtain the total computational complexity (for $h$ small enough)
\begin{equation}\label{com_ml_3}
	\computcom(\{H_l\}_{1\leq l \leq N}) = \widetilde{O}\Big( (C')^d (\frac{1}{h})^{\frac{d}{N}} \sum_{l=0}^{N-1} (\frac{1}{h})^{\frac{l}{N}(1-\nu)d} \Big) = \widetilde{O}\Big( N(C')^d (\frac{1}{h})^{\frac{d}{N}} (\frac{1}{h})^{(1-\nu)d} \Big) \ . 
\end{equation}
Now taking $N = \lfloor \frac{1}{\gamma}  \log(\frac{1}{\varepsilon}) \rfloor$ and using again $h = (C_\gamma^{-1} \varepsilon)^{\frac{1}{\gamma}}$, we get the
formula of~\eqref{complexity_iii}.
\end{proof}

\begin{remark}\label{rk-gamma}
	In~\Cref{complexity_1}, the theoretical complexity bound highly depends on the value of $\gamma \beta$. Recall that $\gamma$ is the convergence rate of fast marching method, as defined in~\Cref{assup_error}. $\beta$,  defined in \Cref{distance_beta}, 
determines the growth of the neighborhood $\mathcal{O}_{\eta}$ of the optimal trajectories, as a function of $\eta$,  
and we already gave there an easy example for which
$\beta=1$, by considering the dynamics with constant speed and a direction chosen from the unit $L^1$ sphere instead of the $L^2$ sphere. 
In particular, one can find examples such that $\gamma=\beta=1$.
	\end{remark}
	
\section{Numerical Experiments}\label{sec-test}

In this section, we present numerical tests on a toy example: Euclidean distance in a box, showing the improvement of our algorithm, compared with the original fast-marching method of Sethian et al.~\cite{sethian1996fast}, using the same update operator. 
Both algorithms were implemented in C++, and executed on a single core of a Quad Core IntelCore I7 at 2.3Gh with 16Gb of RAM.

We start with the easiest case: a constant speed $f(x,\alpha) \equiv 1$ in $\Omega = (0,1)^{d}$. The sets $\sourceset$ and $\destset$ are the Euclidean balls with radius $0.1$, centered at $(0.2,\dots,0.2)$ and $(0.8,\dots,0.8)$ respectively. 
We provide detailed results comparing the performances of
the classical and multi-level fast marching methods. 
In the higher dimension cases, the classical fast marching method 
cannot be executed in a reasonable time. So we fix a time budget of 1 hour.  
\Cref{test1_1} shows the CPU time and the memory allocation in dimensions range from 2 to 6, with grid meshes equal to $\frac{1}{50}$ and $\frac{1}{100}$. 
In all the dimensions, %
we observe that the multi-level fast marching method with finest mesh step $h$ yields the same relative error as the classical fast marching method with mesh step $h$, but with considerably reduced CPU times and memory requirements. 
Moreover, when the dimension is greater than 4, the classical fast marching method could not be executed in a time budget of 1 hour. 
In all, the multi-level method appears to be much less sensitive to the \textit{curse-of-dimensionality}. 

\begin{figure}[htbp]
	\centering 
	\includegraphics[width=0.4\textwidth]{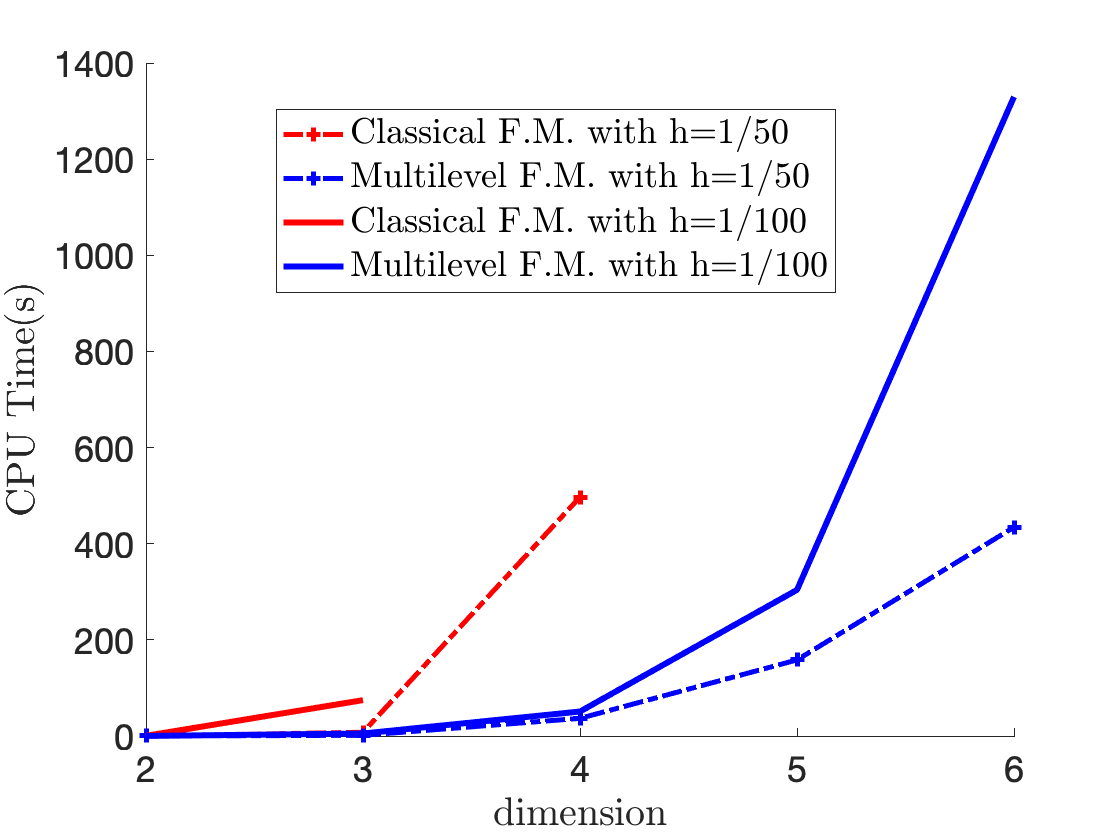} 
	\includegraphics[width=0.4\textwidth]{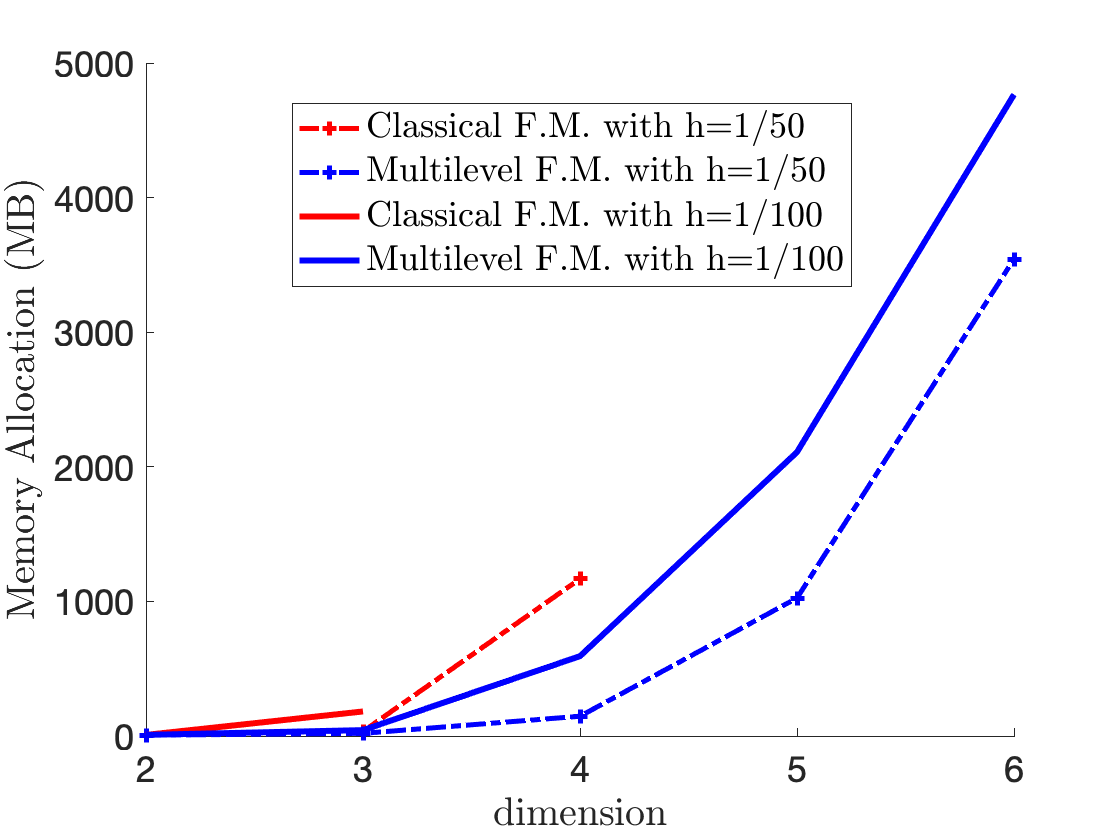}
	\caption{CPU time (left) and memory allocation (right) w.r.t. dimension, $h$ fixed.}
	\label{test1_1}
\end{figure}

To better compare the multi-level method with the classical method, we test the algorithms with several (finest) mesh steps (which are proportional to the predictable error up to some exponent $1/\gamma$), when the dimension is fixed to be $3$, see~\Cref{test1error}. We consider both the 2-level and the multi-level algorithm. In the multi-level case, the number of levels is adjusted to be (almost) optimal.

\begin{figure}[htbp]
	\centering
		\includegraphics[width=0.4\textwidth]{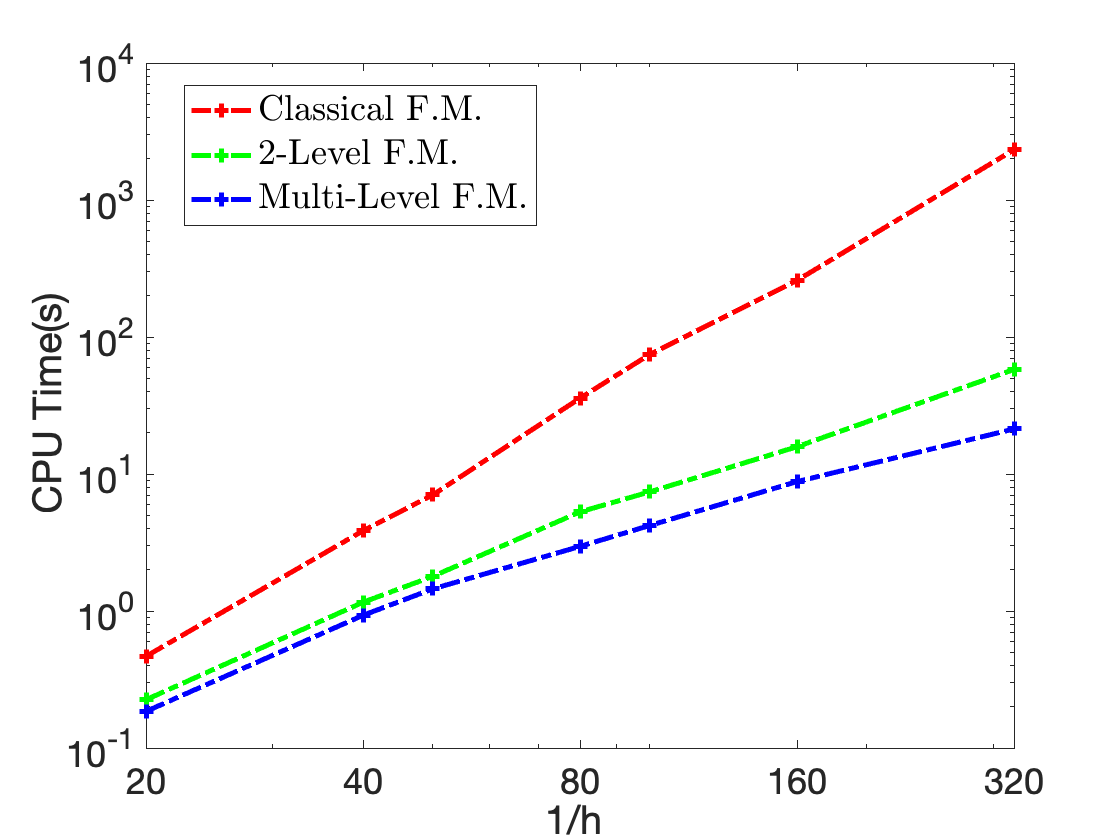}
		\includegraphics[width=0.4\textwidth]{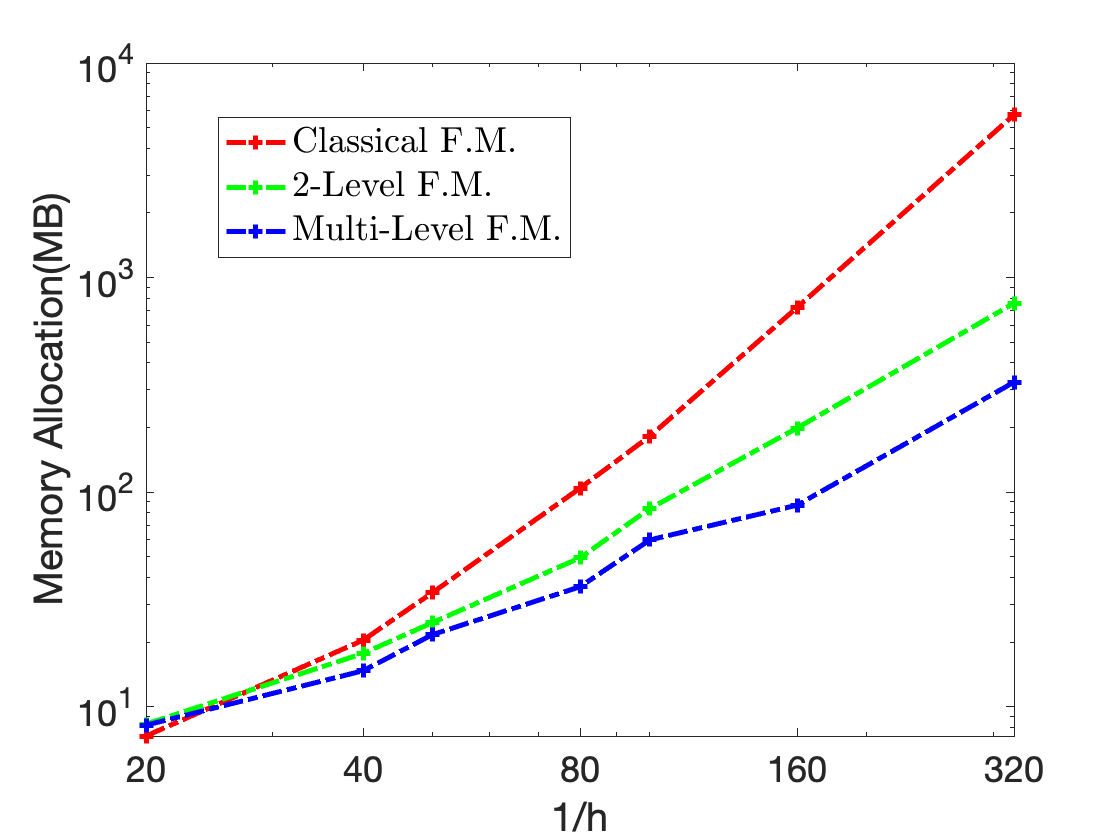}
	\caption{CPU time (left) and memory allocation (right) 
		for several values of the finest mesh step $h$, in dimension $3$, log-log scale.}
	\label{test1error}
\end{figure}

\bibliographystyle{alpha} 

\newcommand{\etalchar}[1]{$^{#1}$}

\end{document}